\documentclass{article}
\usepackage {amsmath, amsfonts, amssymb, mathrsfs, amsthm,accents, sectsty,hyphenat, calc}
\usepackage[usenames]{color}
\usepackage[nohug,midshaft]{diagrams}
\usepackage{palatino}
\def\mf#1{\mathfrak{#1}}
\def\mc#1{\mathcal{#1}}
\def\mb#1{\mathbb{#1}}
\def\tx#1{{\rm #1}}
\def\tb#1{\textbf{#1}}

\def\R{\mathbb{R}}
\def\C{\mathbb{C}}
\def\Q{\mathbb{Q}}

\def\Z{\mathbb{Z}}

\def\lmod{\setminus}

\def\ol#1{\overline{#1}}

\def\rk{\tx{rk}}

\def\hat{\widehat}

\def\rw{\rightarrow}
\def\lw{\leftarrow}

\def\lrw{\longrightarrow}
\def\hrw{\hookrightarrow}

\def\thrw{\twoheadrightarrow}

\def\lw{\leftarrow}
\def\sm{\smallsetminus}

\def\<{\langle}
\def\>{\rangle}

\newenvironment{mytitle}
{\begin{center}\large\sc}
{\end{center}}

\newarrow{Equals}{=}{=}{=}{=}{}

\newtheorem{thm}{Theorem}[section]
\newtheorem{lem}[thm]{Lemma}
\newtheorem{pro}[thm]{Proposition}
\newtheorem{cor}[thm]{Corollary}

\newtheorem{fct}[thm]{Fact}

\sectionfont{\center\sc\normalsize}
\subsectionfont{\bf\normalsize}

\numberwithin{equation}{section}

\begin{document}

\begin{mytitle} Simple wild $L$-packets \end{mytitle}

\section*{Introduction} \label{sec:intro}


The local Langlands correspondence predicts a relationship between the irreducible smooth representations of a $p$-adic group $G$ and the representations of the local Weil-Deligne group into the $L$-group of $G$; these latter representation are usually called Langlands parameters. While this relationship remains conjectural, recently a lot of progress has been made towards establishing it. Led by the framework of conjectures surrounding the local Langlands correspondence, in particular the formal degree conjecture of Hiraga-Ichino-Ikeda, Gross and Reeder study in their recent paper \cite{GrRe10} arithmetic properties of discrete Langlands parameters. They single out a certain class of them, which they call \emph{simple wild parameters}, and determine much of their structure. Their study also points them to a certain class of smooth irreducible representations of $p$-adic groups, which they call \emph{simple supercuspidal representations} and for which they provide an explicit construction. The authors then conjecture that the two classes -- the simple wild parameters and the simple supercuspidal representations -- ought to correspond under the conjectural local Langlands correspondence.

In this paper, under mild restrictions on the residual characteristic of $F$, we explicitly realize the correspondence between simple wild parameters and simple supercuspidal representations. Starting from a simple wild parameter, we construct a finite set of simple supercuspidal representations and then identify it with the group of characters of the (finite) centralizer of the parameter, in accordance with the conjectural structure of $L$-packets. Our construction is compatible with finite unramified extensions of the base field. We are able to show the stability of these $L$-packets for an open subset of regular semi-simple elements without further restrictions on the residue field of $F$. We have learned from Stephen DeBacker that the stability of these $L$-packets for all regular semi-simple elements follows from the recent character formulas of Adler-Spice, under the assumption that the residual characteristic is sufficiently large.

An important motivation for studying simple supercuspidal representations and their Langlands parameters and $L$-packets comes from the special role they play in global applications. In \cite{Gr}, Gross studies the sum of multiplicities of cuspidal automorphic representations of a simple algebraic group $G$ which have a prescribed local behavior at a fixed finite set of places and are unramified elsewhere. The local components of these representations at the finite places in the fixed finite set are set to be either Steinberg or simple supercuspidal. Using the trace formula, Gross shows that the sum of multiplicities of these cuspidal automorphic representations can be expressed by special values of modified Artin L-functions of the motive attached to $G$. When $G$ is simply-connected and defined over a global function field, Gross' work implies that there should be a unique cuspidal automorphic representation whose local components are unramified except at two places, at one of which the component is Steinberg and at the other it is simple supercuspidal. This was later proved by Heinloth, Ngo, and Yun, in \cite{HeNgYu}, where they use this representation to construct an interesting local system on $\mb{G}_m$. The local systems obtained in this way generalize the sheaves constructed by Deligne \cite{SGA4.5}, which geometrize Klooserman sums.

We will now briefly sketch the construction of the simple wild $L$-packets and describe the contents of the paper. Let $F$ be a $p$-adic field with Weil-group $W_F$ and let $G$ be a split simple simply-connected group with complex dual group $\hat G$. The notions of a simple supercuspidal representation and a simple wild parameter are reviewed in Section \ref{sec:simple}. Under the assumption that the residual characteristic $p$ of $F$ does not divide the order of the Weyl group of $G$, Gross and Reeder give a precise analysis of the structure of simple wild parameters. Part of this analysis is summarized at the end of Section \ref{sec:simple}. In this paper, we will impose the weaker condition that $p$ does not divide the Coxeter number of $G$, and work with parameters which satisfy the conditions listed at the end of Section \ref{sec:simple}. Starting from such a parameter
\[ \phi : W_F \rw \hat G \]
we first construct a tamely ramified anisotropic torus $S$ defined over $F$ and a stable conjugacy class of embeddings $S \rw G$. These embeddings, which we call \emph{embeddings of type (C)}, have very special properties. Their study is the backbone of the construction and takes place in Section \ref{sec:embc}. One of the properties of an embedding of type (C) is that the point in the building of $G(F)$ associated to it is the barycenter of an alcove. This already hints at a connection with simple supercuspidal representations, since barycenters of alcoves play a central role in their construction. Further properties include a precise description of the structure of the Moy-Prasad filtration of $S(F)$ as well as a result on the Galois cohomology of $S$, which implies that two stably conjugate embeddings of type (C) are rationally conjugate under the adjoint group $G_\tx{ad}$ of $G$.

The construction of simple supercuspidal $L$-packets is the subject of Section \ref{sec:pack}. Dual to the stable class of embeddings $S \rw G$ is a $\hat G$-conjugacy class of embeddings $^LS \rw \hat G$, where $^LS$ is the $L$-group of the torus $S$. The parameter $\phi$ then factors as
\begin{diagram}
^LS&\rTo^{^Lj}&\hat G\\
\uTo<{\phi_S}&\ruTo^\phi\\
W_F
\end{diagram}
The parameter $\phi_S$ provides a character $\chi : S(F) \rw \C^\times$. The $\hat G$-class of $^Lj$ is not quite canonical, its construction involves an auxiliary choice. However, the results on the structure of $S(F)$ established in Section \ref{sec:embc} allow us to conclude that the $G_\tx{ad}(F)$-conjugacy class of pairs $(j,\chi)$ depends only on the parameter $\phi$. Given such a pair, we consider the quotient $Z(F)G(F)_{x,1/h}/G(F)_{x,2/h}$, where $x$ is the point in the building associated to $j$ and $h$ is the Coxeter number of $G$, and construct from $\chi$ an affine generic character on that quotient. This is the input from which the construction of Gross and Reeder produces a simple supercuspidal representation. In summary, we obtain from the parameter $\phi$ a $G_\tx{ad}(F)$-conjugacy class of pairs $(j,\chi)$ and hence a $G_\tx{ad}(F)$-conjugacy class of simple supercuspidal representations. This latter conjugacy class is the $L$-packet anticipated by Gross and Reeder. We would like to remark that the construction of the affine generic character from $\chi$ is influenced by the work of Adler \cite{Ad98}. In fact, the pair $(j,\chi)$ can be used directly to produce a representation of $G(F)$ via Adler's construction. We have chosen the path through affine generic characters instead in order to show that the packets we obtain are the ones expected by Gross and Reeder.

In Section \ref{sec:int} we study the internal structure of simple wild $L$-packets. More precisely, given a parameter $\phi$, we provide a bijection between the corresponding packet and the Pontryagin dual of the finite abelian group $\tx{Cent}(\phi,\hat G)$. Such a bijection is not unique, it depends on the choice of a Whittaker datum for $G$. The instrumental result in this section is the fact that for a fixed Whittaker datum there is precisely one generic representation in $\Pi_\phi$. This is the statement of the generic packet conjecture -- a property expected to hold for any tempered $L$-packet. Because of its importance, we provide two alternative proofs -- a short one using analytic methods and relying on deep results of Moeglin-Waldspurger and Shelstad, which we have borrowed from the work of DeBacker and Reeder, and a longer but elementary one.

In Section \ref{sec:unram} we study the behavior of our construction under finite unramified extensions of the base field. Gross and Reeder conjecture a certain compatibility with respect to such an extension, and based on that compatibility deduce that simple wild parameters and simple supercuspidal representations should correspond. The result of Section \ref{sec:unram} is that our construction satisfies a compatibility of the expected kind.

The final Section \ref{sec:stab} deals with the stability of the simple wild $L$-packets. We prove, without further restrictions on $F$, that the sum of characters in each $L$-packet is stable on all elements belonging to the image of an embedding of type (C). Moreover, this property does not hold for any subset of a simple wild $L$-packet.

\tb{Acknowledgements:} The author would like to thank Benedict Gross and Mark Reeder for their encouragement and interest in this work, and for stimulating mathematical conversations. The author would also like to thank Stephen DaBacker and Loren Spice for explaining to him the structure of the characters of simple supercuspidal representations and its connection to the stability problem, as well as Thomas Haines and Jared Weinstein for helpful and stimulating conversations. Finally, the support of the National Science Foundation%
\footnote{This material is based upon work supported by the National Science Foundation under agreement No. DMS-0635607. Any opinions, findings and conclusions or recommendations expressed in this material are those of the author and do not necessarily reflect the views of the National Science Foundation.}
is gratefully acknowledged.

\tableofcontents

\section{Notation and preliminaries} \label{sec:notpre}

Throughout the paper, $F$ will denote a $p$-adic field, i.e., a finite extension of $\Q_p$. Its ring of integers and residue field will be $O_F$ and $k_F$. The characteristic of the finite field $k_F$ will be denoted by $p$, and its cardinality by $q$. We will write $\Gamma$ or $\Gamma_F$ for the absolute Galois group of $F$, $W_F \supset I_F \supset P_F$ for the Weil group and its inertia and wild inertia subgroups. Similar notations will be used for any other $p$-adic field $E$, with the appropriate subscript changed. If $E/F$ is a finite extension, we will write $\Gamma(E/F)$, $W(E/F)$, $I(E/F)$ for the relative Galois, Weil, and inertia groups. We will write $v : F^\times \rw \Z$ for the normalized valuation of $F$, and we will write $v : E^\times \rw e(E/F)^{-1}\Z$ for its unique extension to $E$, where $e(E/F)$ is the ramification degree of the extension $E/F$.

Given an algebraic group $G$ defined over $F$, we will write $G(R)$ for the set of points of $G$ with values in an $F$-algebra $R$. The letter $Z$ will denote the center of $G$. More generally, given a subset $S \subset G$, we will write $\tx{Cent}(S,G)$ for the subgroup of $G$ centralizing $S$ and $N(S,G)$ for the subgroup normalizing $S$. The Lie-algebra of $G$ will be denoted by the Fraktur letter $\mf{g}$. When $G$ is semi-simple, we will write $\mc{B}(G,F)$ for the Bruhat-Tits building of $G$ relative to $F$, and $\mc{A}(T,F)$ for the apartment of a given maximal torus $T$. Given a point $x \in \mc{B}(G,F)$ we will write $G(F)_x$ for the full stabilizer of the point $x$ for the action of $G(F)$ on $\mc{B}(G,F)$. Given a real number $r \geq 0$, we will write $G(F)_{x,r}$ for the Moy-Prasad subgroup of $G(F)$. In particular, $G(F)_{x,0}$ is the parahoric subgroup of $G(F)$ associated to $x$, which is a subgroup of finite index in $G(F)_x$. For any real number $r$ we also have the Moy-Prasad lattices $\mf{g}(F)_{x,r}$.

Assume that $G$ is split and $T$ is a split maximal torus. We will write $\Omega(T,G)$ for the Weyl group $N(T,G)/T$. One has $\Omega(T,G)(F)=\Omega(T,G)$. We will write $\Omega_a(T,G)$ for the affine Weyl group $N(T,G)(F)/T(F)_b$, where $T(F)_b$ is the maximal bounded subgroup of $T$. A hyperspecial vertex $o \in \mc{A}(T,F)$ endows both $T$ and $G$ with smooth connected $O_F$-models. When such a vertex is chosen, we will assume this $O_F$-structure understood, i.e. we will reuse the letters $T$ and $G$ for the $O_F$-models of $T$ and $G$. We then have $T(F)_b=T(O_F)$. The subgroup of $\Omega_a(T,G)$ which fixes $o$ will be denoted by $\Omega_a(T,G)_o$. The composition
\[ \Omega_a(T,G)_o \hrw \Omega_a(T,G) \thrw \Omega(T,G) \]
is an isomorphism. As usual, $X^*(T)$ and $X_*(T)$ will be the groups of characters and cocharacters of $T$. We will write $Q \subset P \subset X^*(T) \otimes \Q$ for the lattices of roots and weights, and correspondingly $Q^\vee \subset P^\vee \subset X_*(T)$ for the lattices of coroots and coweights.

We would like to recall the following well-known facts, which we will use.
\begin{fct}
Let $h$ be a positive integer coprime to $p$. The following are equivalent
\begin{enumerate}
\item $F$ has a primitive $h$-th root of unity,
\item $k_F$ has a primitive $h$-th root of unity,
\item there exists a totally ramified Galois extension $E/F$ of degree $h$.
\end{enumerate}
In that case, $E=F(\omega)$ where $\omega$ is a uniformizer in $E$ and $\frac{\sigma(\omega)}{\omega} \in \mu_h(F)$ for all $\sigma \in \Gamma(E/F)$.
\end{fct}

\begin{fct} Let $R$ be a reduced irreducible root system. Then all bad primes, all torsion primes, and all primes dividing the connection index of $R$ also divide the Coxeter number of $R$.
\end{fct}

\section{Simple wild parameters and simple supercuspidal representations} \label{sec:simple}

In this section we would like to give a brief review of simple wild parameters and simple supercuspidal representations, following the work of Gross and Reeder \cite{GrRe10}. Let $G$ be a split, semi-simple, simply-connected group defined over $F$ and let $x \in \mc{B}(G,F)$ be the barycenter of some alcove $C$. Let $h$ denote the Coxeter number of $G$. Assume that $p$ does not divide $h$. One consequence of this assumption is the following: If $o$ is a special vertex in $\mc{B}(G,F)$ then the natural $O_F$-structure on the center $Z$ of $G$ is etale and $Z(F)=Z(O_F)=Z(k_F)$.

The group $G(F)_x$ is an Iwahori subgroup of $G(F)$, and $G(F)_{x,1/h}$ is its pro-unipotent radical. Choose a maximal torus $T$ whose apartment $\mc{A}(T,F)$ contains $C$. Then we have a direct sum decomposition of $k_F$-vector spaces
\[ G(F)_{x,\frac{1}{h}}/G(F)_{x,\frac{2}{h}} = \bigoplus U_\alpha/U_{\alpha+1} \]
where $\alpha$ runs over the $C$-simple affine roots, and $U_\alpha \subset G(F)$ is the corresponding affine root subgroup. Gross and Reeder define \cite[9.2]{GrRe10} an \emph{affine generic character} to be a character
\[ \chi: Z(F)G(F)_{x,\frac{1}{h}} \rw \C^\times \]
whose restriction to each $U_\alpha$ is the inflation of a non-trivial character of $U_\alpha/U_{\alpha+1}$. Given such a character $\chi$, it is shown in \cite[9.3]{GrRe10} that
\[ \pi = \textrm{c-Ind}_{Z(F)G(F)_{x,\frac{1}{h}}}^{G(F)} \chi \]
is an irreducible supercuspidal representation of $G(F)$. These are the \emph{simple supercuspidal representations}. In \cite[\S9.5]{GrRe10}, Gross and Reeder consider the orbits of $G_\tx{ad}(F)$ in the set of these representations. Each such orbit has order equal to that of $Z(F)$, and the authors conjecture that it should constitute an $L$-packet, and moreover the parameter of this $L$-packet should be a simple wild parameter -- an object we will now describe.

Let $\phi : W_F \rw \hat G$ be a continuous homomorphism whose image consists of semi-simple elements. Composing $\phi$ with the adjoint representation we obtain a representation of $W_F$ on the Lie-algebra $\mf{\hat g}$ of $\hat G$. Let us restrict this representation to the inertia subgroup $I_F$. The image of $I_F$ in $\tx{Aut}(\mf{\hat g})$ is a finite group $D_0$, which is the Galois group of a finite extension $L$ of the maximal unramified extension $F^u$ of $F$. This Galois group has the lower ramification filtration
\[ \{1\} = D_n \subset D_{n-1} \subset \dots \subset D_0. \]
The Swan conductor of the representation of $W_F$ on $\mf{\hat g}$ is defined to be the number
\[ b(\phi) := \sum_{j>0} \dim(\mf{\hat g}/\mf{\hat g}^{D_j})\frac{|D_j|}{|D_0|}. \]
The homomorphism $\phi$ is called a \emph{simple wild parameter} if it satisfies the following two conditions:
\begin{itemize}
\item $\phi(I_F)$ has no non-trivial invariants in $\mf{\hat g}$,
\item $b(\phi) = \rk(G)$.
\end{itemize}

Under the assumption that $p$ does not divide the order of the Weyl group, Gross and Reeder \cite[Prop. 5.6 and Prop. 9.4]{GrRe10} carry out a detailed study of the structure of simple wild parameters. The information we are going to need is the following

\begin{itemize}
\item $D_2=\{1\}$,
\item $D_1$ lies in a unique maximal torus $\hat T$ of $\hat G$, in particular
\item $D$ lies in $N(\hat T,\hat G)$,
\item the image of $D_0/D_1$ in $\Omega(\hat T,\hat G)$ is generated by a Coxeter element.
\end{itemize}

Note that a parameter satisfying this list of properties is a simple wild parameter. This is because $\mf{\hat g}^{D_1}=\mf{\hat t}$ and hence $\mf{\hat g}^{D_0}$ is the set of fixed points of a Coxeter element acting on $\mf{\hat t} = X_*(\hat T)\otimes \C$, which is trivial; moreover, the Swan conductor of $\phi$ equals
\[ \dim(\mf{\hat g}/\mf{\hat t})\frac{1}{h} = \frac{|R|}{h} = \rk(G), \]
where $R$ is the root system of $\hat t$ acting on $\mf{\hat g}$.

The assertion of \cite[Prop. 5.6]{GrRe10} is that conversely, if $p$ does not divide the order of the Weyl group, all simple wild parameters satisfy this list of properties.

\section{Embeddings of type (C)} \label{sec:embc}

\subsection{Definition and basic properties}

Let $S$ be a torus, $G$ a reductive group, and $j : S \rw G$ an embedding such that $j(S)$ is a maximal torus of $G$. Assume that all these are defined over $F$. If $j' : S \rw G$ is a second such embedding, we call $j$ and $j'$ stably conjugate if there exists $g \in G(\ol{F})$ such that $j'=\tx{Ad}(g)j$. The map $j$ provides an embedding $Z(G) \rw S$ defined over $F$. This embedding is unchanged if we replace $j$ by $j'$. We will call its image in $S$ again $Z(G)$ and we will write $S_\tx{ad}$ for $S/Z(G)$.
The map $j$ also provides an embedding $\Omega(j(S),G) \rw \tx{Aut}(S)$ defined over $F$. We will call its image $\Omega(S,G)$. When we replace $j$ by $j'$, the subgroup $\Omega(S,G)$ of $\tx{Aut}(S)$ remains unchanged. Note however that if the images of $j$ and $j'$ are the same, the two identifications of $\Omega(j(S),G)=\Omega(j'(S),G)$ with $\Omega(S,G)$ provided by $j$ and $j'$ will in general not be the same. They will differ by conjugation by an element of $\Omega(S,G)(F)$.

Let $E$ be the splitting field of $S$, a finite Galois extension of $F$. We will call a map $j : S \rw G$ an \emph{embedding of type (C)} if the following conditions hold:
\begin{itemize}
\item $S$ is tamely ramified, i.e. $E/F$ is a tamely ramified extension,
\item $G$ is split, semi-simple, and simply-connected,
\item $j$ is an embedding defined over $F$,
\item $j(S)$ is a maximal torus of $G$,
\item the image of $I(E/F)$ in $\tx{Aut}(X^*(S))$ is a cyclic subgroup of $\Omega(S,G)$ generated by a Coxeter element.
\end{itemize}
Following \cite{Pr01}, we can associate to $j$ a point in the Bruhat-Tits building $\mc{B}(G,F)$, namely
\[ \mc{A}(j(S),E)^{\Gamma(E/F)} \subset \mc{B}(G,E)^{\Gamma(E/F)}=\mc{B}(G,F), \]
where the left hand side is a singleton set because $S$ is anisotropic.

Embeddings of type (C) behave well with respect to unramified extensions in the following sense:
\begin{fct} \label{fct:typecbc} Let $j : S \rw G$ be an embedding of type (C) defined over $F$, and let $\tilde F/F$ be an unramified extension. Then $j \times \tilde F : S \times \tilde F \rw G \times \tilde F$ is an embedding of type (C) defined over $\tilde F$. The splitting field of $S \times \tilde F$ is $E\tilde F$. The restriction map
$I(E\tilde F/\tilde F) \rw I(E/F)$ is an isomorphism and respects the inclusions of both groups into $\tx{Aut}(X^*(S))$. Finally, the vertices associated to $j$ and $j \times \tilde F$ coincide under the natural inclusion $\mc{B}(G,F) \rw \mc{B}(G,\tilde F)$.
\end{fct}
\begin{proof} Clear. \end{proof}

\begin{lem} \label{lem:trivcoh} Let $S \rw G$ be an embedding of type (C). Then the map $H^1(F,S) \rw H^1(F,S_\tx{ad})$ is trivial. \end{lem}
\begin{proof}
By Tate-Nakayama duality this is equivalent to the statement that the map $H^1(A,Q) \rw H^1(A,P)$ is trivial, where $Q \subset P$ are the root and weight lattice of the root system of $G$, and $A \subset A(R)$ is a subgroup containing a Coxeter element $c$. Let $B \subset A$ be the cyclic subgroup generated by $c$. Since $P^B=0$, the restriction maps provide isomorphisms
\[ H^1(A,Q) \rw H^1(B,Q)^{B/A} \qquad H^1(A,P) \rw H^1(B,P)^{B/A} \]
Hence it is enough to prove that the map $H^1(B,Q) \rw H^1(B,P)$ is trivial. This follows from the fact that for any $q \in Q$ there exists $p \in P$ with $q=p-cp$.
\end{proof}

\begin{cor} \label{cor:adstab} Let $j : S \rw G$ be an embedding of type (C). Then $G_\tx{ad}(F)$ acts transitively on the stable class of $j$. \end{cor}
\begin{proof}
Let $j'$ be stably conjugate to $j$ and $g \in G(\ol{F})$ be such that $j'=\tx{Ad}(g)j$. Then $g^{-1}\sigma(g) \in H^1(F,j(S))$. Applying lemma \ref{lem:trivcoh} we can find $t \in j(S_\tx{ad})(\ol{F})$ such that $t^{-1}\sigma(t)=g^{-1}\sigma(g)$. Then $gt^{-1} \in G_\tx{ad}(F)$ and $j'=\tx{Ad}(gt^{-1})j$.
\end{proof}

\subsection{Properties of the source of an embedding of type (C)}

Let again $S$ be an $F$-torus and let $E$ be its splitting field, which we assume to be tamely ramified. In \cite[\S3]{ChYu01} the authors introduce the so called \emph{finite-type Neron model} of $S$. It is a smooth $O_F$-group scheme of finite type with generic fiber equal to $S$, whose group of $O_F$-points equals the bounded subgroup of $S(F)$.

Since the generic fiber of the Neron model of $S$ coincides with $S$, it will cause no confusion to use the letter $S$ both for the torus and its Neron model. Let $S^0$ be the connected component of the Neron model. Then $S^0(O_F)$ is the Iwahori subgroup of $S(F)$.

We call $S$ inertially anisotropic if there are no $I(E/F)$-fixed vectors in $X^*(S)$. This is the case in particular if $S$ is the source of an embedding of type (C). By \cite[10.2.1]{BLR90}, such an $S$ has a Neron model in the classical sense. In fact, all the following integral models coincide for such a torus: The classical Neron model, the lft-Neron model, the ft-Neron model, the standard model.

The topological group $S(F)$ comes equipped with the Moy-Prasad filtration (see \cite[3.2]{MoPr96} or \cite[4.2]{Yu03}), which is defined as follows:
\[ S(F)_r = \{ s \in S^0(O_F)|\forall \chi \in X^*(S): v(\chi(s)-1) \geq r \} \]

Notice that, since the compact group $S(O_F)$ is mapped by any $\chi \in X^*(S)$ into $O_E^\times$, the condition $v(\chi(s)-1) \geq 0$ is vacuous and we have
\[S(F)_0=S^0(O_F).\]
On the other hand, if we assume that $E/F$ is tamely ramified, the argument of \cite[4.7.2]{Yu03} shows that any $s \in S(F)$ satisfying $v(\chi(s)-1)>0$ for all $\chi \in X^*(S)$ actually lies in $S^0(O_F)$. This shows that for $r>0$
\[ S(F)_r = \{ s \in S(F) | \forall \chi \in X^*(S) : v(\chi(s)-1) \geq r \}. \]
One also has a filtration on the Lie algebra $\mf{s}$ of $S$ defined for all $r \in \R$ by
\[ \mf{s}(F)_r = \{ X \in \mf{s}(F) | \forall \chi \in X^*(S) : v(d\chi(s)) \geq r \}. \]
Here the Neron model does not play a role.

\begin{pro} \label{pro:sfilt} Let $S' \rw G$ be an embedding of type (C) and denote the Coxeter number of $G$ by $h$. Let $S$ be a quotient of $S'$ by any subgroup of $Z(G)$. Then
\begin{itemize}
\item $S(F)_0=S(F)_\frac{1}{h}$
\item For $0<i<h$, the quotients
      \[ S(F)_\frac{i}{h}/S(F)_\frac{i+1}{h}\qquad\tx{and}\qquad \mf{s}(F)_\frac{i}{h}/\mf{s}(F)_\frac{i+1}{h} \]
      are canonically isomorphic $k_F$-vector spaces. Their dimension is equal to the multiplicity with which $i$ appears as an exponent of the root system of $G$.
\item For $i=1$ and $i=h-1$, the above $k_F$-vector spaces are one-dimensional.
\end{itemize}
\end{pro}
\begin{proof}
Recall that $E$ denotes the splitting field of $S$ and let $\omega$ be a uniformizer of $E$.

We begin with the first statement. Let $s \in S^0(O_F)$ and $\chi \in X^*(S)$. A-priori we know $\chi(s) \in O_E^\times$. Using \cite[Theorem 1.3]{NaXa91} and \cite[Theorem 2.2]{KuSa01} we see that the special fiber of $S^0$ is a product of Witt groups. Thus there exists a power $p^n$ such that $s^{p^n}$ has trivial image in $S^0(k_F)$. It follows from \cite[4.7.3]{Yu03} that $s^{p^n} \in S(F)_1$, thus $\chi(s)^{p^n} \in 1+\omega O_E$. In other words, the image of $\chi(s)^{p^n}$ in $k_E^\times$ is trivial, but then already the image of $\chi(s)$ in $k_E^\times$ must be trivial. This is equivalent to
$v(\chi(s)-1) \geq v(\omega)=e(E/F)^{-1}=h^{-1}$.

Turning to the second statement, fix $0<i<h$ and consider the sequences
\[ 1 \rw 1+\omega^{i+1} O_E \rw 1+\omega^i O_E \rw k_E \rw 0, \]
and
\[ 0 \rw \omega^{i+1} O_E \rw \omega^i O_E \rw k_E \rw 0. \]
The map $1+\omega^i O_E \rw k_E$ is given by $1+\omega^ix \mapsto [x]$, where $[\ ]: O_E \rw k_E$ is the reduction map. Similarly, the map $\omega^iO_E \rw k_E$ is given by $\omega^ix \mapsto [x]$.
Let
\[ f : \Gamma(E/F) \rw k_E^\times, \qquad \gamma \mapsto \left[\frac{\gamma(\omega)}{\omega}\right]^i \]
Then $f \in Z^1(\Gamma(E/F),k_E^\times)$ and after twisting the usual action of $\Gamma(E/F)$ on $k_E$ via $f$ the above exact sequences become sequences of $\Gamma(E/F)$-modules.
Tensoring with the $\Gamma(E/F)$-module $X_*(S)$ and taking $\Gamma(E/F)$-fixed points we obtain the two sequences
\[ 1 \rw S(F)_\frac{i+1}{h} \rw S(F)_\frac{i}{h} \rw H^0(\Gamma(E/F),X_*(S) \otimes k_E ) \rw H^1(\Gamma(E/F),S(E)_\frac{i+1}{h}),  \]
\[ 0 \rw \mf{s}(F)_\frac{i+1}{h} \rw \mf{s}(F)_\frac{i}{h} \rw H^0(\Gamma(E/F),X_*(S) \otimes k_E ) \rw H^1(\Gamma(E/F),\mf{s}(E)_\frac{i+1}{h}).  \]
where we have used that $H^0(\Gamma(E/F),S(E)_r)=S(F)_r$ for $r>0$, according to \cite[4.7.2]{Yu03}.
We will show that the maps
\begin{equation} \label{eq:filtiso} S(F)_\frac{i}{h} / S(F)_\frac{i+1}{h} \rw H^0(\Gamma(E/F),X_*(S) \otimes k_E ) \lw \mf{s}(F)_\frac{i}{h} / \mf{s}(F)_\frac{i+1}{h} \end{equation}
are isomorphisms of $k_F$-vector spaces and that the dimension of the middle space equals the multiplicity of $i$ as an exponent of the root system of $G$. Notice that although $i$ is not present in the notation for the middle space, it influences it because we have taken a twisted action of $\Gamma(E/F)$ on $k_E$, and the twist depends on $i$.

By construction it is clear that the maps in \eqref{eq:filtiso} are $k_F$-linear and injective. The surjectivity of the first map will follow if we show that $H^1(\Gamma(E/F),S(E)_r)$ is trivial. To that end, consider the inflation-restriction sequence
\[ H^1(\Gamma(E/F)/I(E/F),S(F')_r) \rw H^1(\Gamma(E/F),S(E)_r) \rw H^1(I(E/F),S(E)_r), \]
where $F':=E^{I(E/F)}$. The last group in that sequence is trivial, since $I(E/F)$ is a finite group whose order is prime to $p$, while $S(E)_r$ is an abelian pro-p group. The argument of \cite[2.3.1]{DeRe09} shows that the first group is also trivial. We have thus shown that the first map in \eqref{eq:filtiso} is an isomorphism.

Next we consider the group $H^0(\Gamma(E/F),X_*(S) \otimes k_E)$. Let $m$ be the multiplicity of $i$ as an exponent of $R$. We will first show that $H^0(I(E/F),X_*(S) \otimes k_E)$ is a $k_E$-vector space of dimension $m$. Fix a generator $\sigma \in I(E/F)$. Then $\sigma(\omega)=\zeta\omega$, where $\zeta \in O_E^\times$ is an element whose $E/F'$-norm is $1$. In fact,
we may assume that $\omega$ is chosen so that $\zeta \in O_{F'}^\times$ is a primitve $h$-th root of unity, and then so is $[\zeta] \in k_E^\times$. Furthermore, by assumption, $\sigma$ acts on $X^*(S)$ via a Coxeter element $c$. With this notation, $H^0(I(E/F),X_*(S) \otimes k_E)$ is the $k_E$-subspace of $X_*(S) \otimes k_E$ fixed by the action of $c \otimes [\zeta]^i$, which is the same as the $[\zeta]^{-i}$-eigenspace for the action of $c$ on $X_*(S)\otimes k_E$. The dimension of this space is equal to the highest power of $(X-[\zeta]^{-i})$ dividing the image in $k_E[X]$ of the characteristic polynomial $P_c \in \Z[X]$ of the action of $c$ on $X_*(S)$. By Hensel's lemma this is the same as the highest power of $(X-\zeta^{-i})$ dividing the image of $P_c$ in $O_{F'}[X]$. This power remains the same when we consider divisibility in $\ol{F}[X]$ instead of $O_{F'}[X]$. We can choose a field homomorphism $\ol{F} \rw \C$ which sends $\zeta^{-1}$ to $\exp(\frac{2\pi i}{h})$ and consider divisibility in $\C[X]$. But the highest power of $(X-\exp(\frac{2\pi i}{h}))$ which divides $P_c$ in $\C[X]$ is by definition equal to $m$, see \cite[V.6]{Bou02}.

We have shown that $H^0(I(E/F),X_*(S) \otimes k_E)$ is a $k_E$-vector space of dimension $m$. It is equipped with an action of $\Gamma(E/F)/I(E/F)=\tx{Gal}(k_E/k_F)$ which is compatible with the natural action of this group on $k_E$. Thus the set of fixed points is a $k_F$-vector space of dimension $m$.

To complete the proof of the second statement of the proposition, we have to show that the second map in \eqref{eq:filtiso} is an isomorphism. Since we already know that it is injective, it will be enough to compare the dimensions of its source and target. These are both vector spaces over a finite field, so this is equivalent to comparing their orders as abstract groups. By what we have just proved, it is the same as comparing the orders of the first and third term in \eqref{eq:filtiso}. These are equal due to \cite[5.6]{Yu03}.

The third statement is a direct corollary of the second and \cite[VI.1.11.30]{Bou02}.

\end{proof}

\begin{pro} \label{pro:sprod} Let $S \rw G$ be an embedding of type (C). Then
\[ S(F) = S(F)_0 \times Z(G)(F). \]
\end{pro}
\begin{proof}
By assumption $p$ does not divide the Coxeter number of $G$. In particular $p$ does not divide the order of the finite diagonalizable group $Z(G)$. Thus $Z(G)(F)$ is an abelian group of order prime to $p$, while $S(F)_0$ is a pro-p-group by Proposition \ref{pro:sfilt}. This shows
\[ S(F)_0 \cap Z(G)(F) = \{1\}. \]
In particular, the projection $S(F) \rw S(F)/S(F)_0$ restricts to an injection
\[ Z(G)(F) \rw S(F)/S(F)_0 \]
and we need to show that this injection is also surjective. Since both groups are finite, it is enough to compare orders.

We begin by computing the order of $Z(G)(F)$. We have $Z(G)(F)=Z(G)(O_F)$, with respect to the $O_F$-structure on $G$ provided by the vertex $o$. The $O_F$-group scheme $Z(G)$ is etale, thus $Z(G)(O_F)=Z(G)(k_F)$. We know that
\[ X^*(Z(G))=P/Q, \]
where $P$ and $Q$ are the weight and root lattices for $T$ in $G$. The quotient $P/Q$ is a finite abelian group and we have the corresponding decompositions
\[ P/Q = \prod \Z/n_i\Z\qquad\tx{and}\qquad Z(G) = \prod \mu_{n_i} \]
with all $n_i$ coprime to $p$. We conclude
\[ |Z(G)(F)| = \prod_i \gcd(q-1,n_i). \]

Next we turn to the order of $S(F)/S(F)_0$. We have the bijection
\[ S(F)/S(F)_0 = S(O_F)/S^0(O_F) \rw S(k_F)/S^0(k_F) \rw \pi_0(S)(k_F). \]
Let $\sigma$ be the Frobenius automorphism of $\ol{k_F}$. Since $\pi_0(S)$ is an etale group scheme over $k_F$, the target of the above isomorphism equals $\pi_0(S)(\ol{k_F})^\sigma$.
According to \cite{Xa93} we have
\[ \pi_0(S)(\ol{k_F})^\sigma = \tx{Hom}_\Z(H^1(I,X^*(S)),\Q/\Z)^\sigma \]
and this group has the same order as $H^1(I,X^*(S))_\sigma$. Choosing an admissible isomorphism $S \rw T$ defined over $\ol{F}$ we can identify $X^*(S)$ with $P$ and under this identification $\Gamma$ acts on $P$ through a finite subgroup of the Weyl group. The image of $I$ in this subgroup is a cyclic subgroup $C$ generated by a Coxeter element $c \in C$. The inflation map provides an isomorphism
\[ H^1(C,P) \rw H^1(I,X^*(S)). \]
We claim that the map
\[ H^1(C,P) \rw H^1(C,P/Q) \]
is also an isomorphism. The injectivity follows from the fact that $H^1(C,Q) \rw H^1(C,P)$ is trivial, a fact we already used in the proof of Lemma \ref{lem:trivcoh}. The surjectivity follows from the fact that by periodicity, $H^2(C,Q)=H^0(C,Q)$ and the latter is trivial due to the ellipticity of $c$.

Both isomorphisms are equivariant for the action of $\Gamma$. Now $\Gamma$ acts trivially on $P/Q$, while $\sigma$ acts on $C$ by $c\mapsto c^q$. It follows that the isomorphism of groups
\[ H^1(C,P/Q) = \tx{Hom}(C,P/Q) \rw P/Q,\qquad \xi \mapsto \xi(c) \]
transports the action of $\sigma$ on the left to multiplication by $q^{-1}$. What we are looking for is then the order of $H_0(q^{-1},P/Q)$. The group $P/Q$ being finite, this is the same as the order of
\[ H^0(q^{-1},P/Q) = H^0(q,P/Q) = \tx{Ker}(q-1|P/Q) \]
and it is readily checked that the order of the latter is given by the same formula as the formula for the order of $Z(G)(F)$ provided earlier.
\end{proof}

\subsection{A splitting}

\begin{pro} \label{pro:bary}
Let $j: S \rw G$ be an embedding of type (C). Then the associated point in $\mc{B}(G,F)$ is the barycenter of an alcove.
\end{pro}

\begin{proof}
Fix a splitting $(T,B,\{X_\alpha\})$ of $G$ over $\Z$, and let $C \subset \mc{A}(T,F)$ and $o \in \ol{C}$ be the corresponding alcove and special vertex. Let $x := \mc{A}(j(S),E)^{\Gamma(E/F)}$ be the point of $\mc{B}(G,F)$ associated to the $j$. We may conjugate $j$ under $G(F)$ to achieve $x \in \ol{C}$, which we will henceforth assume. We will show that $x$ is the barycenter of $C$.

Recall that $E$ denotes the splitting field of $S$. According to Fact \ref{fct:typecbc} we may, without loss of generality, base change to the maximal unramified subextension of $E/F$ and henceforth assume that $E/F$ is totally ramified. Choose a generator $\sigma \in \Gamma(E/F)$.
Let $p \in G(E)$ be such that $\tx{Ad}(p)T=S$. The image of $p^{-1}\sigma(p) \in N(T,G)(E)$ in the affine Weyl group $\Omega_a(T,G)$ can be written as $\lambda w$ with $\lambda \in Q^\vee$ and $w \in \Omega_a(T,G)_o$. By assumption, $w$ projects to a Coxeter element in $\Omega(T,G)$, thus there exists $\mu \in P^\vee$ with $\lambda=w\mu-\mu$. Put $q := p\mu(\omega) \in G_\tx{ad}(E)$. A direct calculation reveals that the image of $q^{-1}\sigma(q) \in N(T,G_\tx{ad})(E)$ in the affine Weyl group $\Omega_a(T_\tx{ad},G_\tx{ad})$ equals $w$. Being elliptic, the unique fixed point of $w$ in $\mc{A}(T,E)$ is $o$. Thus we have an element $q \in G_\tx{ad}(E)$ with $\tx{Ad}(q)T=S$ and $qo=x$.

Write $q=th$ with $t \in T_\tx{ad}(E)$ having the property $to=x$ and $h \in G_\tx{ad}(E)_o$. Consider the exact sequence
\[ 1 \rw O_E^\times \rw E^\times \stackrel{-v}{\lrw} e^{-1}\Z \rw 1, \]
where $e=[E:F]$. Tensoring with $P^\vee=X_*(T_\tx{ad})$ we obtain
\[ 1 \rw T_\tx{ad}(O_E) \rw T_\tx{ad}(E) \rw e^{-1}P^\vee \rw 1. \]
The image of $t$ in $e^{-1}P^\vee$ is the translation of $\mc{A}(T,E)$ sending $o$ to $x$. The proposition will be proved once we show that this image is
\[ \xi := \sum_{\alpha \in \Delta} e^{-1}\check\omega_\alpha, \]
where $\Delta$ is the set of simple roots for $(T,B)$ and $\check\omega_\alpha$ is the fundamental coweight corresponding to $\alpha$. In fact, since we are already assuming that $x \in \ol{C}$, it will be enough to show that the image of $t \in e^{-1}P$ lies in the same $\Omega_a(T_\tx{ad},G_\tx{ad})$-orbit as $\xi$.

Taking $\Gamma(E/F)$-fixed points in the last exact sequence we obtain an isomorphism
\[ e^{-1}P^\vee/P^\vee \rw H^1(E,T_\tx{ad}(O_E)). \]
The kernel of the reduction map $T_\tx{ad}(O_E) \rw T_\tx{ad}(k_F)$ is an abelian pro-p-group, while $\Gamma(E/F)$ has prime-to-p order. It follows that $H^1(E,T_\tx{ad}(O_E)) \rw H^1(E,T_\tx{ad}(k_F))$ is injective. The action of $\Gamma(E/F)$ on $T_\tx{ad}(k_F)$ is trivial, hence evaluation at $\sigma$ provides an isomorphism
\[ H^1(E,T_\tx{ad}(k_F)) \rw T_\tx{ad}(k_F)[e], \]
the latter group being the group of $e$-torsion points of $T_\tx{ad}(k_F)$. Composing these maps we obtain the $\Omega(T_\tx{ad},G_\tx{ad})$-equivariant injection
\[ e^{-1}P^\vee/P^\vee \rw T_\tx{ad}(k_F)[e]. \]
The image of $\xi$ under this injection equals
\[ \sum_{\alpha \in \Delta} \check w_\alpha\left[\frac{\omega}{\sigma(\omega)}\right], \]
while image of $t$ under the composition of $T_\tx{ad}(E) \rw e^{-1}P^\vee$ with this injection equals $[t^{-1}\sigma(t)]$. We seek to show that those two elements of $T_\tx{ad}(k_F)[e]$ lie in the same Weyl-orbit. Let us call them $\bar\xi$ and $\bar t$. It is clear that $\bar\xi$, considered as an element of $G_\tx{ad}(k_F)$, is regular. Its order is $e$, which is equal to the Coxeter number of $G_\tx{ad}$ since $j$ is an embedding of type (C). But it is known that there is a unique $G(\ol{k_F})$-conjugacy class of regular elements of that order \cite[III.2.12]{SpSt70}. Of course this class is semi-simple in our case, as $p$ and $e$ are coprime. In particular there is a unique Weyl-orbit of regular elements in $T_\tx{ad}(k_F)[e]$. Hence our task is to show that $\bar t$ is regular.

Recall that
\[ q^{-1}\sigma(q) = h^{-1}t^{-1}\sigma(t)\sigma(h) \]
belongs to $N(T_\tx{ad},G_\tx{ad})(O_E)$ and is a lift of a Coxeter element. The image of this element under the reduction map $G_\tx{ad}(O_E) \rw G_\tx{ad}(k_F)$ is equal to $\bar h\bar t \bar h^{-1}$, where $\bar h$ is the image of $h$ under the reduction map $\tilde G_\tx{ad}(O_E) \rw \tilde G_\tx{ad}(k_F)$, and $\tilde G_\tx{ad}$ is the Bruhat-Tits $O_E$-group scheme whose $O_E$-points are $G_\tx{ad}(E)_o$. This is still a lift of a Coxeter element, hence a regular element \cite[III.2.3]{SpSt70}.
\end{proof}

Let $j : S \rw G$ be an embedding of type (C), $x$ its associated point in $\mc{B}(G,F)$, and $C$ the alcove whose barycenter is $x$. Let $dj : \mf{s} \rw \mf{g}$ be the differential of $j$. By \cite[1.9.1]{Ad98} we have

$dj(\mf{s}(F)_r)=dj(\mf{s}(F))_r$ for all $r$. Then for $r=\frac{1}{h},\frac{2}{h}$ we have by \cite[1.9.3]{Ad98}
\begin{equation} \label{eq:perp} \mf{g}(F)_r = dj(\mf{s}(F))_r \oplus [dj(\mf{s}(F))^\perp]_{x,r}, \end{equation}
where $\perp$ is taken with respect to the Killing form on $\mf{g}(F)$ and $[\ ]_{x,r}$ means intersection with $\mf{g}(F)_{x,r}$. It should be noted that the statement in loc. cit. refers not to the Killing form, but to what is called there a ``good'' bilinear form, which is a form constructed in \cite[\S4]{AdRo00}. Such a good bilinear from exists in our case, because the assumptions of \cite[\S4]{AdRo00} are implied by the existence of an embedding of type (C), as remarked in Section \ref{sec:notpre}. The group $G$ being simple, a good bilinear form is just a scalar multiple of the Killing form, and hence provides the same notion of perpendicularity.

Projecting onto the first factor in the decomposition \eqref{eq:perp} and composing with $dj^{-1}$ we obtain a homomorphism of $k_F$-vector spaces
\begin{equation} \label{eq:afgenproj} \mf{g}(F)_{x,\frac{1}{h}}/\mf{g}(F)_{x,\frac{2}{h}} \rw \mf{s}(F)_\frac{1}{h}/\mf{s}(F)_\frac{2}{h} \end{equation}

\begin{pro} \label{pro:afgenproj} Let $\alpha$ be an affine simple root with respect to the alcove $C$. The restriction of \eqref{eq:afgenproj} to the subspace
\[ \mf{g}_\alpha(F)_{x,\frac{1}{h}} / \mf{g}_\alpha(F)_{x,\frac{2}{h}} \]
is non-trivial.
\end{pro}
\begin{proof}
Let $o$ denote a special vertex contained in $\ol{C}$ and $T$ be a split maximal torus whose apartment contains $C$. Let $B$ be a bilinear form on $\mf{g}(F)$ as constructed in \cite[\S4]{AdRo00}. It restricts to a non-degenerate bilinear form
\begin{equation*} \label{eq:bbar} \bar B : \mf{g}(F)_{x,-\frac{1}{h}}/\mf{g}(F)_{x,0} \times \mf{g}(F)_{x,\frac{1}{h}}/\mf{g}(F)_{x,\frac{2}{h}} \rw k_F
\end{equation*}
The two arguments of $\bar B$ have the following direct sum decompositions
\begin{eqnarray} \label{eq:decomp}
\mf{g}(F)_{x,\frac{1}{h}}/\mf{g}(F)_{x,\frac{2}{h}}&=&\mf{g}_{-\eta}(F)_{o,1}/\mf{g}_{-\eta}(F)_{o,2}\oplus \bigoplus_{\alpha \in \Delta} \mf{g}_\alpha(F)_{o,0} / \mf{g}_\alpha(F)_{o,1}\nonumber\\
\mf{g}(F)_{x,-\frac{1}{h}}/\mf{g}(F)_{x,0}&=&\mf{g}_{\eta}(F)_{o,-1}/\mf{g}_{\eta}(F)_{o,0}\oplus \bigoplus_{\alpha \in -\Delta} \mf{g}_\alpha(F)_{o,0} / \mf{g}_\alpha(F)_{o,1}
\end{eqnarray}
where $\Delta$ is the set of simple roots and $\eta$ is the highest root. Put $\Pi=\Delta \cup \{-\eta\}$ and for $\alpha \in \Pi$ let $\mf{\bar g}_\alpha$ resp. $\mf{\bar g}_{-\alpha}$ denote the corresponding constituents of the first resp. second decomposition in \eqref{eq:decomp}.

Inside the first argument of $\bar B$ we have the subspace
\[ \mf{s}(F)_{-\frac{1}{h}}/\mf{s}(F)_0 \]
embedded via $dj$. This subspace is one-dimensional due to Proposition \ref{pro:sfilt} and the fact that multiplication by an uniformizer of $F$ provides an isomorphism $\mf{s}(F)_r \rw \mf{s}(F)_{r+1}$ for any $r \in \R$.

The proof of the current proposition will be complete if we show that this one-dimensional subspace is not $\bar B$-orthogonal to $\mf{\bar g}_\alpha$ for any $\alpha \in \Pi$. But the orthogonal complement of a given $\mf{\bar g}_\alpha$ is precisely the direct sum of $\mf{\bar g}_\beta$ for $-\beta \in \Pi \sm \{\alpha\}$. Hence it will be enough to find an element of $\mf{s}(F)_{-\frac{1}{h}}/\mf{s}(F)_0$ whose coordinate for each $-\beta \in \Pi$ is non-trivial.

There are two reductions we need to make. First, notice that we are free to replace $F$ by any finite unramified extension. Second, we are going to show that we are free to replace $j$ with any stable conjugate as long as the associated point in $\mc{B}(G,F)$ remains unchanged. To see why this is true, recall first that by Corollary \ref{cor:adstab}, this will replace $j$ by a conjugate under $G_\tx{ad}(F)_x$. By the invariance of $B$ we may shift the conjugation from the left argument of $\bar B$ to its right argument. The action of $G_\tx{ad}(F)_x$ on the right argument of $\bar B$ factors through $G_\tx{ad}(F)_x/G_\tx{ad}(F)_{x,\frac{1}{h}}$ and one easily sees that this action preserves the decomposition \eqref{eq:decomp}.

With these reductions in place, we are now going to construct a certain finite unramified extension $\tilde F$ of $F$, a certain stable conjugate $j'$ of $j$, and then proceed to construct the sought element of $\mf{s}(F)_{-\frac{1}{h}}/\mf{s}(F)_0$.

The unramified extension $\tilde F$ is constructed as follows. Let $F' \subset E$ be the maximal unramified subextension of $F$, $\sigma \in \Gamma(E/F')$ a generator and $\omega \in E$ a uniformizer such that $\sigma(\omega)\omega^{-1}=\zeta_h$ is a root of unity in $F'$ of order $h$. Put
\[ \bar c_s := \prod_{\alpha \in \Delta} \check\omega_\alpha(\zeta_h^{-1}) \in T_\tx{ad}(O_{F'}), \]
There exists a lift $c_s \in T(O_{F^u})$ of $\bar c_s$.

Let $\tilde F$ be such that $c_s \in T(O_{\tilde F})$.

To ease notation, we now replace $F$ by $\tilde F$. Then $E$ is replaced by $E\tilde F$. The group $\Gamma(E/F)$ remains unchanged and we keep the notations $\sigma$,$\omega$ and $\zeta_h$.

Next, we construct the appropriate stable conjugate of $j$. Let $c \in N(T,G)(O_F)$ be a lift of a Coxeter element. According to \cite[III.2.12]{SpSt70}, the elements $c$ and $c_s$ are conjugate under $G(\ol{F})$, hence by \cite[7.1]{Ko86} also under $G(O_F)$. Let $h \in G(O_F)$ be such that $\tx{Ad}(h)c=c_s$. Put
\[ t := \prod_{\alpha \in \Delta} \check\omega_\alpha(\omega^{-1}) \in T_\tx{ad}(E). \]
Then $t^{-1}\sigma(t)=c_s$. Put $p := th \in G_\tx{ad}(E)$. Then we have
\[ p^{-1}\sigma(p) = c, \qquad \tx{and}\qquad \tx{Ad}(p)c=c_s. \]
We can choose an admissible isomorphism $S \rw T$ which transports the action of $\sigma$ on $X^*(S)$ to the action of $c$ on $X^*(T)$. Composing this isomorphism with $\tx{Ad}(p)$ we obtain an admissible embedding $j' : S \rw G$ of type (C). The unique fixed point for the action of $\tx{Ad}(c)\sigma$ on $\mc{A}(T,E)$ is $o$, hence the unique fixed point for the action of $\Gamma(E/F)$ on $\mc{A}(j'(S),E)$ is $po=to=x$. The embedding $j'$ is the stable conjugate of $j$ that we were looking for. Again, we ease notation by replacing $j$ by $j'$.

Before we continue, recall the following simple fact: If
\[ \mf{g} = \mf{t} \oplus \bigoplus_\alpha \mf{g}_\alpha \]
is the root decomposition of a semi-simple Lie algebra, $\Delta$ is a set of simple roots and $\Pi$ is the union of $\Delta$ and the negative of the highest root, then an element
\[ X \in \bigoplus_{\alpha \in \Pi} \mf{g}_\alpha \]
is semi-simple only if all of its coordinates are non-zero. This statement is \cite[Lemma 7.2]{Kt59}, where it is proved for complex Lie algebras, but the proof clearly does not depend on the complex field.

We will now construct the special element $X \in \mf{s}(F)_{-\frac{1}{h}} \sm \mf{s}(F)_0$. Let $\tilde X_0 \in \mf{t}(F)$ be any eigenvector of $\tx{Ad}(c)$ with eigenvalue $\zeta_h$. Multiplying by an appropriate power of a uniformizer of $F$ we may arrange that $\tilde X_0 \in \mf{t}(F)_0 \sm \mf{t}(F)_{0+}$. Put $X_0 := \omega^{-1}\tilde X_0$, so that $X_0 \in \mf{t}(E)_{-\frac{1}{h}} \sm \mf{t}(E)_0$. By construction we have
\[ \tx{Ad}(c)X_0=\zeta_h X_0, \qquad \tx{and}\qquad \tx{Ad}(c)\sigma(X_0)=X_0, \]
so that $X:=\tx{Ad}(p)X_0$ is an element of $\mf{s}(F)_{-\frac{1}{h}} \sm \mf{s}(F)_0$ which is an eigenvector of $\tx{Ad}(c_s)$ with eigenvalue $\zeta_h$. Recalling the explicit form of $c_s$ and the fact that the hight of the highest root is $h-1$, the standard root decomposition of $\mf{g}(F)$ shows that $X$ belongs to the subspace
\[ \bigoplus_{\alpha \in -\Pi} \mf{g}_\alpha(F). \]
We know that the projection of $\omega X_0$ to $\mf{g}(E)_{o,0} / \mf{g}(E)_{o,\frac{1}{h}}$ is a semi-simple element in this $k_E$-Lie algebra, thus the projection of $\omega X$ to $\mf{g}(E)_{x,0} / \mf{g}(E)_{x,\frac{1}{h}}$ is a semi-simple element as well. The fact that we recalled above now implies that for each $\alpha \in -\Delta$, the coordinate of $\omega X$ in $\mf{g}_\alpha(E)$, which a-priori belongs to $\mf{g}_\alpha(E)_{x,0}=\mf{g}_\alpha(E)_{o,\frac{1}{h}}$, does not belong to $\mf{g}_\alpha(E)_{x,\frac{1}{h}}=\mf{g}_\alpha(E)_{o,\frac{2}{h}}$, while for $\alpha$ being the highest root, the corresponding coordinate belongs to $\mf{g}_\alpha(E)_{o,\frac{1-h}{h}} \sm \mf{g}_\alpha(E)_{o,\frac{2-h}{h}}$. Using the fact that $X$ and each $\mf{g}_\alpha$ are defined over $F$, we see the image of $X$ in $\mf{g}(F)_{x,-\frac{1}{h}}/\mf{g}(F)_{x,0}$ has non-zero coordinates (with respect to the decomposition \eqref{eq:decomp}) for all $-\alpha \in \Pi$.

\end{proof}

The following proposition can be proved by the very same argument. We will not need it in this paper, but it may be of independent interest.

\begin{pro} Let $j : S \rw G$ be an embedding of type (C). Let $C$ be the alcove determined by $j$ and $o$ a hyperspecial vertex in $\ol{C}$. Let $dj : \mf{s} \rw \mf{g}$ be the differential of $j$. Then the image under $dj$ of every element of $\mf{s}(F)_\frac{1}{h} \sm \mf{s}(F)_\frac{2}{h}$ projects to a regular unipotent element in $\mf{g}(k_F)$.
\end{pro}

\subsection{Automorphisms of an alcove}

Let $j : S \rw G$ be an embedding of type (C), and let $x \in \mc{B}(G,F)$ be the point associated to it. By Proposition \ref{pro:bary}, $x$ is the barycenter of some alcove $C$. Then the action of $S_\tx{ad}(F)$ on $\mc{B}(G,F)$ preserves $C$. On the other hand, there is a natural action of $G_\tx{ad}(F)$ on $C$: For any $g \in G_\tx{ad}(F)$ we choose $h \in G(F)$ such that $hgC=C$. Then $hg$ is an automorphism of $C$ and does not depend on the choice of $h$, because any element of $G(F)$ which preserves $C$ fixes it pointwise.

\begin{pro} \label{pro:caut} Every automorphism of $C$ coming from $\tx{G}_\tx{ad}(F)$ can be realized by an element of $S_\tx{ad}(F)$.
\end{pro}

Before we can prove this Proposition, we need some preparation. We begin by recalling some basic facts about the Kottwitz homomorphism \cite[\S7]{Ko97}. Let $T \subset G$ be a split maximal torus whose apartment contains $C$. Let $Q^\vee$ and $P^\vee$ be its coroot and coweight lattices, and put $\Lambda = P^\vee/Q^\vee$. The Kottwitz homomorphism for $G_\tx{ad}$ is a surjective homomorphism
\[ G_\tx{ad}(F) \rw \Lambda \]
whose restriction to any parahoric subgroup is trivial, and whose restriction to $T_\tx{ad}(F)$ is given by
\[ \lambda(\pi) \mapsto \lambda, \qquad \forall \lambda \in P^\vee \]
where $\pi \in F$ is any uniformizer. The group $\Lambda$ acts on the alcove $C$, and Kottwitz's homomorphism intertwines the actions of $G_\tx{ad}(F)$ and $\Lambda$ on $C$.

\begin{lem} \label{lem:kotalt} There exists a bijection
\[ \Lambda \rw H^1(I,Z) \]
whose composition with the Kottwitz homomorphism equals
\begin{equation} \label{eq:kotalt} G_\tx{ad}(F) \rw H^1(\Gamma,Z) \rw H^1(I,Z) \end{equation}
which is the composition of the boundary homomorphism and the restriction map.
\end{lem}
\begin{proof}
Fix a hyperspecial vertex $o \in \ol{C}$ and consider the Cartan decomposition
\[ G_\tx{ad}(F) = G_\tx{ad}(O_F)T_\tx{ad}(F)G_\tx{ad}(O_F). \]
The Kottwitz homomorphism is the unique surjective map $G_\tx{ad}(F) \rw \Lambda$ which is trivial on $G_\tx{ad}(O_F)$ and whose restriction to $T_\tx{ad}(F)$ is given by
\[ T_\tx{ad}(F) \rw T_\tx{ad}(F)/T_\tx{ad}(O_F) \rw P^\vee \rw P^\vee/Q^\vee. \]
We are going to show that
\begin{enumerate}
\item The map \eqref{eq:kotalt} is trivial on $G_\tx{ad}(O_F)$.
\item The restriction of \eqref{eq:kotalt} to $T_\tx{ad}(F)$ is surjective and factors through $T_\tx{ad}(F) \rw \Lambda$.
\item The finite groups $\Lambda$ and $H^1(I,Z)$ have the same size.
\end{enumerate}
First, consider the sequence of $\Gamma(F^u/F)$-modules
\[ 1 \rw Z(O_{F^u}) \rw G(O_{F^u}) \rw G_\tx{ad}(O_{F^u}) \rw 1. \]
It is exact due to lemma \ref{lem:hensel}, and gives rise to the bottom row of the diagram
\begin{diagram}
G_\tx{ad}(F)&\rTo&H^1(F,Z)\\
\uTo&&\uTo\\
G_\tx{ad}(O_F)&\rTo&H^1(F^u/F,Z(O_{F^u}))
\end{diagram}
in which the left vertical map is the natural inclusion and the right vertical map is the inflation map. Since the residual characteristic of $F$ does not divide the exponent of $Z$, we have $Z(O_{F^u}) = Z$, and from the inflation-restriction sequence we obtain that \eqref{eq:kotalt} is indeed trivial on $G_\tx{ad}(O_F)$. This shows point 1, as well as the claim that \eqref{eq:kotalt} factors through $T_\tx{ad}(F) \rw P^\vee$. Since \eqref{eq:kotalt} also annihilates the image of $T(F)$ in $T_\tx{ad}(F)$, we obtain that it factors through $T_\tx{ad}(F) \rw \Lambda$.

The surjectivity of \eqref{eq:kotalt} upon restriction to $T_\tx{ad}(F)$ follows from the surjectivity of the connecting homomorphism $T_\tx{ad}(F) \rw H^1(F,Z)$, which is due to $T$ being split, and from the surjectivity of the restriction homomorphism $H^1(F,Z) \rw H^1(I,Z)$, which we now argue. The group $H^2(F^u,Z)$ is trivial, because $Z$ is a product of finitely many $\mu_n$ for $n$ prime to the residual characteristic of $F$.
It follows that the image of the restriction map is equal to the group of fixed points of Frobenius acting on $H^1(I,Z)$. If $I_t$ denotes the tame quotient of inertia, then the inflation map $H^1(I_t,Z) \rw H^1(I,Z)$ is a Frobenius-equivariant bijection. Hence we are looking for the Frobenious-fixed points on $H^1(I_t,Z)=\tx{Hom}(I_t,Z)$. But Frobenius acts on both $I_t$ and $Z$ as multiplication by $q$ (if we think of $I_t$ and $Z$ additively for a moment), and so its action on $\tx{Hom}(I_t,Z)$ is trivial. This shows point 2.

For point 3 we only need to observe that since $I_t$ is pro-cyclic and acts trivially on $Z$, choosing any topological generator of $I_t$ provides a bijection from $H^1(I_t,Z)$ to $Z$, but the size of $Z$ is equal to the size of $\Lambda$.
\end{proof}

\begin{lem} \label{lem:hensel} Let $J$ be any connected reductive group defined over $O_F$. If $p$ does not divide $|\pi_0(Z(J))|$,
then $J(O_{F^u}) \rw J_\tx{ad}(O_{F^u})$ is surjective.
\end{lem}
\begin{proof} Let $Z$ be the center of $J$. The morphism $J \rw J_\tx{ad}$ is surjective, hence flat \cite[Exp 6b, 3.11]{SGA3}. Thus it is an fpqc cover. On the other
hand, the same morphism is $Z$-invariant, and is in fact an fpqc-torsor under $Z$ -- to trivialize it, use it itself as a trivializing cover. Our assumption
implies that $Z$ is smooth over $O_F$ \cite[Exp 9,4.10]{SGA3}. Since smoothness descends over fpqc maps \cite[17.7.3]{EGA4}, we conclude that $J \rw J_\tx{ad}$
is smooth. The claim now follows from the variant of Hensel's lemma for smooth morphisms: Let $x \in J_\tx{ad}(O_{F^u})$. There exists a finite unramified
extension $E/F$ such that $x \in J_\tx{ad}(O_E)$. Let $\bar x \in J_\tx{ad}(k_E)$ be its reduction. Enlarging $E$ if necessary we can find a lift $\bar y \in
J(k_E)$ of $\bar x$. We claim that there exists $x \in J(O_E)$ which maps to $\bar x$ and $y$ respectively. Using the lifting property of smooth morphisms we
can find elements $x_n \in J(O_E/\pi^n O_E)$ for every $n$. The family $(x_n)$ is compatible with the reduction mod $\pi$ maps and the claim follows from the
completeness of $O_E$.
\end{proof}

We now proceed to prove Proposition \ref{pro:caut}.
\begin{proof}
We need to show that the restriction of the Kottwitz isomorphism to $S_\tx{ad}(F)$ is surjective. By Lemma \ref{lem:kotalt}, this is the same as showing that the restriction of \eqref{eq:kotalt} to $S_\tx{ad}(F)$ is surjective. Consider the diagram
\begin{diagram}
S(F)&\rTo&S_\tx{ad}(F)&\rTo&H^1(F,Z)\\
\dTo&&\dTo&&\dTo>{\tx{Res}}\\
S(F^u)&\rTo&S_\tx{ad}(F^u)&\rTo&H^1(I,Z)\\
\end{diagram}
The restriction of $S(F) \rw S_\tx{ad}(F)$ to the $0+$ filtration subgroups of both sides is surjective. To see this, consider the map $X_*(S) \rw X_*(S_\tx{ad})$. It is injective and its cokernel is torsion of exponent prime to $p$. Let $E$ be the splitting field of $S$ and $\omega$ a uniformizer. Tensoring with the pro-p group $(1+\omega O_E)$ we obtain the bijection $S(E)_{0+} \rw S_\tx{ad}(E)_{0+}$. Taking Galois invariants we arrive at the claim. Applying Proposition \ref{pro:sprod} we see that $S_\tx{ad}(F)/S_\tx{ad}(F)_{0+}$ injects into $H^1(F,Z)$.

One can upgrade this argument to show that $S_\tx{ad}(F^u)/S_\tx{ad}(F^u)_{0+}$ injects into $H^1(I,Z)$ -- the bijectivity of $S(F^u)_{0+} \rw S_\tx{ad}(F^u)_{0+}$ follows in the same way as before, and to obtain the decomposition $S(F^u)=Z(F^u) \times S(F^u)_{0+}$ we just need to observe that $S(F^u)$ is the direct limit of $S(F')$ over all finite unramified extensions $F'$ of $F$, and then apply Proposition \ref{pro:sprod} to each $S(F')$ (which is possible by Fact \ref{fct:typecbc}) and observe that the inclusions between the various $S(F')$ respect these decompositions of $S(F')$ into its pro-p part and prime-to-p part. We note that the injection $S_\tx{ad}(F^u)/S_\tx{ad}(F^u)_{0+} \rw H^1(I,Z)$ is in fact also surjective, because $H^1(I,S)$ is trivial due to a theorem of Steinberg.

It remains to show that the natural map
\[ S_\tx{ad}(F)/S_\tx{ad}(F)_{0+} \rw S_\tx{ad}(F^u)/S_\tx{ad}(F^u)_{0+} \]
is surjective. Applying Proposition \ref{pro:sfilt} with $F$ replaced by an arbitrary unramified extension, we see that this map is the top horizontal map of the following diagram.
\begin{diagram}
S_\tx{ad}(O_F)/S_\tx{ad}^\circ(O_F)&&\rTo&&S_\tx{ad}(O_{F^u})/S_\tx{ad}^\circ(O_{F^u})\\
\dTo&&&&\dTo\\
\pi_0(S_\tx{ad})(k_F)&&\rTo&&\pi_0(S_\tx{ad})(\ol{k_F})
\end{diagram}
We claim that the two vertical arrows are isomorphisms. To see this, consider the exact sequence of etale sheaves over $\tx{Spec}(O_F)$
\[ 1 \rw S_\tx{ad}^\circ \rw S_\tx{ad} \rw \pi_0(S_\tx{ad}) \rw 1. \]
It gives us an exact sequence of $\Gamma(F^u/F)$-modules
\[ 1 \rw S_\tx{ad}^\circ(O_{F^u}) \rw S_\tx{ad}(O_{F^u}) \rw \pi_0(S_\tx{ad})(O_{F^u}) \rw 1, \]
and using the fact that $H^1(F^u,S_\tx{ad}^\circ(O_{F^u}))=1$, we obtain the exact sequence
\[ 1 \rw S_\tx{ad}^\circ(O_F) \rw S_\tx{ad}(O_F) \rw \pi_0(S_\tx{ad})(O_F) \rw 1. \]
The claim now follows from the fact that the group scheme $\pi_0(S_\tx{ad})$ is etale over $O_F$.

To complete the proof of the proposition, it is enough to show that the action of $\Gamma(F^u/F)$ on $\pi_0(S_\tx{ad})(\ol{k_F})$ is trivial. According to \cite{Ra05}, there is an isomorphism of $\Gamma(F^u/F)$-modules
\[ \pi_0(S_\tx{ad})(\ol{k_F}) \rw X_*(S_\tx{ad})_I. \]
We have $X^*(S_\tx{ad})=P^\vee$, where $P^\vee$ is the coweight lattice of $S$ in $G$. Let $c \in \tx{Aut}_\Z(P^\vee)$ be the image of a topological generator of $I$. By assumption $c$ is a Coxeter element. By the theory of the Coxeter element, we have $[P^\vee]_I=P^\vee/Q^\vee$, where $Q^\vee$ is the set of coroots of $S$ in $G$. Since Frobenius acts on $P^\vee$ through the Weyl group, and the Weyl group acts trivially on $P^\vee/Q^\vee$, the proof is complete.
\end{proof}

\section{From simple wild parameters to $L$-packets} \label{sec:pack}

Let $G$ be a split, simple, simply-connected group defined over $F$. Assume that the residual characteristic of $F$ does not divide the Coxeter number of $G$. Let $\phi : W_F \rw \hat G$ be a Langlands parameter satisfying the conditions listed at the end of Section \ref{sec:simple}. According to \cite[5.6]{GrRe10}, these are all simple wild parameters if $p$ does not divide the order of the Weyl group of $G$. Then $\hat T:=\tx{Cent}(\phi(P_F),\hat G)^\circ$ is a maximal torus in $\hat G$ and the image of $\phi$ is contained in $N(\hat T, \hat G)$. The composition of $\phi$ with the projection $N(\hat T,\hat G) \rw \Omega(\hat T,\hat G)$ is a tamely-ramified homomorphism
\[ w: W_F \rw \Omega(\hat T,\hat G) \]
Let $^LS := \hat T \rtimes_w W_F$ and let $S$ be the $F$-torus with $L$-group $^LS$. There is a unique stable conjugacy class $[j]$ of embeddings $j : S \rw G$. The conditions on $G$ and $\phi$ imply that these are embeddings of type (C).

Choose tamely ramified $\chi$-data for ${^LS}$ and let $[^Lj] : {^LS} \rw \hat G$ be the $\hat G$-conjugacy class of $L$-embeddings corresponding to it \cite[2.6]{LS87}. Choose one element $^Lj$ in this conjugacy class with the properties that $^Lj(\hat S)=\hat T$ and that the homomorphism
\[W_F \rTo^{^Lj} N(\hat T,\hat G) \rOnto \Omega(\hat T,\hat G) \]
is equal to $w$. Then the image of $\phi$ is contained in the image of $^Lj$ and we obtain a factorization
\[ \phi : W_F \rTo^{\phi_S} \,^LS \rTo^{^Lj} \hat G \]
of $\phi$. Let $\chi_S : S(F) \rw \C^\times$ be the character corresponding to $\phi_S$.

\begin{lem} \label{lem:ltd} Let $S$ be a tamely ramified torus with splitting field $E$. Let
\[ \phi : W_{E/F} \rw \hat S \rtimes W_{E/F},\qquad \phi(w) = (\phi_0(w),w) \]
be a Langlands-parameter and $\chi : S(F) \rw \C^\times$ the corresponding character. Then, for $r>0$, we have
\[ \phi_0|_{E^\times_r} = 1\qquad \Longrightarrow\qquad \chi|_{S(F)_r}=1 \]
\end{lem}
\begin{proof}
This argument is a strengthening of \cite[\S5]{Ha93}. We have the following commutative diagram, where all cohomology and Hom-groups are assumed continuous,
\begin{diagram}
H^1(W_{E/F},\hat S)&\rTo^{\tx{Res}}&\tx{Hom}(E^\times, \hat S)\\
\uTo<{\tx{Cor}}&&\uEquals\\
\tx{Hom}(E^\times,\hat S)_{\Gamma(E/F)}&\rTo^N&\tx{Hom}(E^\times,\hat S)\\
\uEquals&&\uEquals\\
\tx{Hom}(S(E),\C^\times)_{\Gamma(E/F)}&&\tx{Hom}(S(E),\C^\times)\\
\dTo^{\tx{Res}}&\ruTo\\
\tx{Hom}(S(F),\C^\times)
\end{diagram}
Here $\tx{Cor}$ is the corestriction map $H^1(E^\times,\hat S) \rw H^1(W_{E/F},\hat S)$ which is surjective and whose kernel is the augmentation ideal for the action of $\Gamma(E/F)$; the horizontal $\tx{Res}$ is the restriction map on cohomology; $N$ is the norm map for the action of $\Gamma(E/F)$; the vertical $\tx{Res}$ is restriction of characters; and finally the diagonal map is composition with the norm map $S(E) \rw S(F)$. We refer the reader to \cite{La85} for these statements, in particular to 5.6 and 5.7 there.

The left vertical isomorphism is the local Langlands correspondence for tori. The right vertical isomorphism restricts to an isomorphism
\[ \tx{Hom}(E^\times_r,\hat S) \cong \tx{Hom}(S(E)_r,\C^\times). \]
Our assumption thus implies that the restriction of $\chi \circ N_{E/F}$ to $S(E)_r$ is trivial. Let $F' \subset E$ be the maximal unramified subextension. First we claim that the restriction of $\chi \circ N_{F'/F}$ to $S(F')_r$ is trivial. In fact, since $S(F')_r \subset S(E)_r$, we know that the restriction of $\chi\circ N_{E/F}$ to $S(F')_r$ is trivial. This restriction equals $(\chi\circ N_{F'/F})^{[E:F']}$. But $S(F')_r$ is a pro-p-group, while $[E:F']$ is coprime to $p$, which implies that already $\chi\circ N_{F'/F}$ is trivial on $S(F')_r$, as claimed. But this completes the proof, because according to \cite[5.1]{Re08} $N_{F'/F}$ is a surjection $S(F')_r \rw S(F)_r$.
\end{proof}

\begin{pro} \label{pro:chiindep} The $\Omega(S,G)(F)$-orbit of $\chi_S$ is independent of all choices. Its restriction to $S(F)_\frac{2}{h}$ is trivial.\end{pro}
\begin{proof}
Let us first prove independence. We need to examine the choices of $\chi$-data and of a representative ${^Lj}$ within its $\hat G$-class. Keeping first the $\chi$-data fixed, we have the freedom of replacing $^Lj$ by $\tx{Ad}(g)\circ{^Lj}$ with $g \in N(\hat T,\hat G)$ such that its image in $\Omega(\hat T,\hat G)$ commutes with the image of $w$. But then there exists $u \in \Omega(S,G)(F)$ such that $\tx{Ad}(g)\circ{^Lj}={^Lj}\circ\hat u$. Thus replacing $^Lj$ by $\tx{Ad}(g)\circ{^Lj}$ replaces $\chi_S$ by $\chi_S\circ u$.

Now we examine the influence of $\chi$-data. According to Proposition \ref{pro:sprod} it is enough to show that the restrictions of $\chi_S$ to $Z(G)(F)$ and $S(F)_0$ remain unaffected by any change of $\chi$-data. By \cite[2.6.3]{LS87}, changing the $\chi$-data results in changing the $L$-embedding $^Lj$ to ${^Lj'} = c \otimes {^Lj}$ where $c : W_{E/F} \rw \hat T$ is the cocycle constructed in \cite[2.5.B]{LS87}. This has the effect of multiplying $\chi_S$ by the character of $S(F)$ corresponding to $c$, which we will call $\chi_c$.

The restriction of $\chi_c$ to $Z(G)(F)$ corresponds to the image of $c$ under
\[ H^1(W_F,\hat T) \rw H^2(W_F,Z(\hat G)), \]
where on $\hat T$ we have taken the action of $W_F$ provided via $^Lj$. But by construction $c$ takes values in $\hat T_\tx{sc}$, hence its image in $H^2(W_F,Z(\hat G))$ is trivial.

Next we will consider the restriction of $\chi_c$ to $S(F)_0$. For $x \in E^\times$, the value of $c(x)$ is given by
\[ \prod_\lambda \prod_\sigma \zeta_\lambda( N_{E/F_{+\lambda}}(\sigma x))^{\sigma^{-1}\lambda}, \]
where $\lambda$ runs over a set of representatives for the symmetric orbits of $\Gamma$ in $R^\vee(S,G)$, $\sigma$ runs over a set of representatives for the quotient $\Gamma_{E/F}/\Gamma_{E/F_{\pm\lambda}}$, and $\Gamma_{E/F_{+\lambda}}\subset\Gamma_{E/F_{\pm\lambda}}$ are the subgroups of $\Gamma_{E/F}$ fixing $\lambda$ or leaving invariant the set $\{\lambda,-\lambda\}$ respectively, and $\zeta_\lambda$ is the character of $E^\times$ which measures the difference between the two different $\chi$-data. What is important is that since both $\chi$-data are tamely-ramified, so are the characters $\zeta_\lambda$. It follows that the restriction of $c$ to $E^\times_{0+}$ is trivial, which by Lemma \ref{lem:ltd} implies that the corresponding character is trivial on $S(F)_{0+}$. According to Proposition \ref{pro:sfilt} we have $S(F)_{0+}=S(F)_0$, and we conclude that the restriction of $\chi_S$ to $S(F)_0$ is unaffected by the change of $\chi$-data.

Finally we consider the restriction of $\phi_S$ to $E^\times_\frac{2}{h}$. Since the $\chi$-data is tamely-ramified, the same argument as in the previous paragraph shows that this restriction equals the restriction of the original parameter $\phi$ to $E^\times_\frac{2}{h}$. We claim that this restriction is trivial, which again by Lemma \ref{lem:ltd} would complete the proof. Let $L$ be the fixed field of the kernel of $\phi$. The subgroup $W_L/W_E^c$ of $W_{E/F}$ corresponds to the subgroup $N_{L/E}(L^\times)$ of $E^\times$. The explicit knowledge of the lower numbering of $\Gamma(L/E)$ allows us to conclude using \cite[V.6., Cor.3]{Se79} that $E^\times_\frac{2}{h}$ belongs to $N_{L/E}(L^\times)$ and the claim follows.
\end{proof}

Fix a representative $j$ in the stable class $[j]$ such that the associated point in $\mc{B}(G,F)$ is the barycenter $x$ of $C$. We obtain a homomorphism
\begin{eqnarray} \label{eq:homgen}
Z(F)G(F)_{x,\frac{1}{h}}/G(F)_{x,\frac{2}{h}}&\cong&Z(F)\mf{g}(F)_{x,\frac{1}{h}}/\mf{g}(F)_{x,\frac{2}{h}}\nonumber\\\
&\rw&Z(F)\mf{s}(F)_\frac{1}{h}/\mf{s}(F)_\frac{2}{h}\\
&\cong&S(F)/S(F)_\frac{2}{h},\nonumber
\end{eqnarray}
where the first isomorphism is the canonical Moy-Prasad isomorphism, the last isomorphism is given by Propositions \ref{pro:sprod} and \ref{pro:sfilt}, and the middle homomorphism is \eqref{eq:afgenproj}. Composing this homomorphism with $\chi_S$ we obtain a character
\[ \chi_j : Z(F)G(F)_{x,\frac{1}{h}}/G(F)_{x,\frac{2}{h}} \rw \C^\times, \]
which according to Proposition \ref{pro:afgenproj} is affine generic. Put
\[ \pi_{\chi_j} := \textrm{c-Ind}_{Z(F)G(F)_{x,\frac{1}{h}}}^{G(F)} \chi_j \]
This is a simple supercuspidal representation of $G(F)$, as defined in \cite[9.3]{GrRe10} and reviewed in Section \ref{sec:simple}. Define
\[ \Pi_\phi := \tx{Ad}(G_\tx{ad}(F))\cdot \pi_{\chi_j} \]
This is a finite set of representations with cardinality $|Z(F)|$. What remains to be shown is that it only depends on $\phi$. This follows from the next Proposition.

\begin{pro}
The $G_\tx{ad}(F)_x$-conjugacy class of $\chi_j$ is independent of all choices.
\end{pro}
\begin{proof}
The choices we need to examine are those of $j$ and $\chi_S$.

If we replace $j$ be a stable conjugate $j'$ whose associated vertex is still $x$, then by Corollary \ref{cor:adstab} the embeddings $j$ and $j'$ are conjugate under $G_\tx{ad}(F)_x$.

According to Proposition \ref{pro:chiindep}, we may replace $\chi_S$ by a conjugate by an element of $\Omega(S,G)(F)$. We have the diagram
\begin{diagram}
1&\rTo&S(F)&\rTo&N(S,G)(F)&\rTo&\Omega(S,G)(F)&\rTo&H^1(F,S)\\
&&\dTo&&\dTo&&\dEquals&&\dTo\\
1&\rTo&S_\tx{ad}(F)&\rTo&N(S_\tx{ad},G_\tx{ad})(F)&\rTo&\Omega(S_\tx{ad},G_\tx{ad})(F)&\rTo&H^1(F,S_\tx{ad})\\
\end{diagram}
with exact rows, where we have identified $S$ with its image under $j$. Lemma \ref{lem:trivcoh} implies that the map
\[ N(S_\tx{ad},G_\tx{ad})(F)\rTo\Omega(S_\tx{ad},G_\tx{ad})(F) \]
is surjective. But it is clear that $N(S_\tx{ad},G_\tx{ad})(F) \subset G_\tx{ad}(F)_x$.
\end{proof}

\section{The internal structure of simple wild $L$-packets} \label{sec:int}

In this section we will establish a bijection
\begin{equation} \label{eq:ps} X^*(\tx{Cent}(\phi,\hat G)) \rw \Pi_\phi. \end{equation}
This bijection will depend on the choice of a Whittaker datum $(B,\psi)$ for $G$, and will send the trivial character to a $(B,\psi)$-generic representation. It is one of the conjectural properties of $L$-packets -- the generic packet conjecture -- that each $L$-packet should contain precisely one representations which is generic with respect to a fixed Whittaker datum. We are going to verify this conjecture for our packets, that is, we are going to prove the following.

\begin{pro} \label{pro:gen} Without any assumptions on the residue field $k_F$, each $G_\tx{ad}(F)$-orbit of simple supercuspidal representations has a unique element $\pi$ with
\[ \tx{Hom}_{B_u(F)}(\pi,\psi) \neq 0. \]
In particular, this is true for the $L$-packets $\Pi_\phi$.
\end{pro}

This Proposition turns out to also be the key to the construction of the bijection \eqref{eq:ps}. Namely, we will construct a simply transitive action of $X^*(\tx{Cent}(\phi,\hat G))$ on $\Pi_\phi$ which is independent of any choices. Then each Whittaker datum will provide, according to Proposition \ref{pro:gen}, a unique base point in $\Pi_\phi$, and the bijection \eqref{eq:ps} will be the corresponding orbit map.

To construct the action, write $\phi={^Lj}\circ\phi_S$ as in the construction of $\Pi_\phi$. Then ${^Lj}$ provides a bijection
\[ \tx{Cent}(\phi,\hat G) = \tx{Cent}({^Lj},\hat G) \rw \hat S^\Gamma \]
and hence a bijection
\begin{equation} \label{eq:psi1} X^*(\tx{Cent}(\phi,\hat G)) \rw X^*(\hat S^\Gamma) = X_*(S)_\Gamma = H^1(F,S). \end{equation}
We claim that this bijection is independent of the choice of factorization $\phi={^Lj}\circ\phi_S$. Indeed, $^Lj$ is determined up to conjugation under $\Omega(S,G)(F)$. If $c \in \Omega(S,G)$ is a generator for the image of inertia, then $\Omega(S,G)(F) \subset \tx{Cent}(c,\Omega(S,G))$. But $c$ is by assumption a Coxeter element, hence it generates its own centralizer in $\Omega(S,G)$. It follows that $\Omega(S,G)(F)$ acts trivially on $X_*(S)_\Gamma$.

Now consider the maps
\begin{equation} \label{eq:psi2} G_\tx{ad}(F)/G(F) \rw H^1(F,Z) \rw H^1(F,S). \end{equation}
The first map is the boundary homomorphism associated to the obvious sequence. It is an isomorphism of groups, surjectivity following from Kneser's vanishing theorem. The second map is induced from the inclusion $Z \rw S$. It is surjective due to Lemma \ref{lem:trivcoh}. Its kernel is the image of $S_\tx{ad}(F)$. Composing \eqref{eq:psi1} and \eqref{eq:psi2} we obtain an isomorphism of abelian groups
\begin{equation} \label{eq:intiso} X^*(\tx{Cent}(\phi,\hat G)) \rw G_\tx{ad}(F)/G(F)S_\tx{ad}(F). \end{equation}
Here we have mapped $S_\tx{ad}$ into $G_\tx{ad}$ via $j$, but since the $G_\tx{ad}(F)$-conjugacy class of $j$ is unique and the quotient $G_\tx{ad}/G$ is abelian, this image does not depend on the choice of $j$. Via this bijection, the group on the left acts in a natural way on $\Pi_\phi$, and the action is transitive. It is also simple, due to the following lemma.

\begin{lem} Let $j : S \rw G$ be an embedding of type (C), $\chi_S : S(F)/S(F)_\frac{2}{h} \rw \C^\times$ a character, and $\chi$ the affine generic character obtained by composing $\chi_S$ with \eqref{eq:homgen}. Then stabilizer in $G_\tx{ad}(F)$ of the isomorphism class of the representation $\textrm{c-Ind}\chi$ is precisely $G(F)S_\tx{ad}(F)$.
\end{lem}
\begin{proof}
Let $x \in \mc{B}(G,F)$ be the point associated go $j$. Let $C$ be the alcove containing $x$, $o$ a hyperspecial vertex contained in $\ol{C}$, and $T$ a split maximal torus whose apartment contains $C$.

Let $g \in G_\tx{ad}(F)$ be such that $\pi\circ\tx{Ad}(g) \cong \pi$, where $\pi=\textrm{c-Ind}_{Z(F)G(F)_{x,\frac{1}{h}}}^{G(F)}\chi$. We can modify $g$ by an element of $G(F)$ to assume that it fixes the point $x$. Applying \cite[Proposition 9.3.(2)]{GrRe10}, we can modify $g$ by an element of $T(O_F)$ and assume that it now stabilizes the affine generic character $\chi$. The group $S_\tx{ad}(F)$ fixes both $x$ and $\chi$, and by Proposition \ref{pro:caut}, we can modify $g$ by an element of $S_\tx{ad}(F)$ to assume that it fixes all of $x$, $o$, and $\chi$. Hence $g$ belongs to the stabilizer of $\chi$ in the Iwahori-subgroup $G_\tx{ad}(F)_{x,0}$. The subgroup $G_\tx{ad}(F)_{x,1/h}$ acts trivially on $\chi$, and $T_\tx{ad}(O_F) \subset G_\tx{ad}(F)_{x,0}$ surjects onto the quotient $G_\tx{ad}(F)_{x,0}/G_\tx{ad}(F)_{x,1/h}$, so we may assume $g \in T_\tx{ad}(O_F)$. But the genericity of $\chi$ now implies $g=1$.
\end{proof}

\subsection{Proof of Proposition \ref{pro:gen} -- existence}
In this subsection only, we drop the requirement that $\tx{char}(k_F)$ is does not divide the Coxeter number of $G$.

We will now prove the existence statement in Propisition \ref{pro:gen}. Let $\pi:=\textrm{c-Ind} \chi$, where $\chi : Z(F)G(F)_{x,1/h} \rw \C^\times$ is affine generic. Let $B$ be a Borel subgroup containing a maximal torus whose apartment contains $x$, and consider the Whittaker datum $(B,\chi|_{B_u(F)})$. Applying \cite{Ku77} we see that the non-trivial group $\tx{Hom}_{B_u(F)\cap G(F)_{x,1/h}}(\chi,\chi)$ is a direct summand of
$\tx{Hom}_{B_u(F)}(\pi,\chi)$ -- it corresponds to the image of $1$ in
\[ U(F)\lmod G(F)/G(F)_{x,1/h}. \]
This shows that $\pi$ is generic with respect to $(B,\chi|_{B_u(F)})$. Since $G_\tx{ad}(F)$ acts transitively on the set of Whittaker data, we conclude that for each Whittaker datum $(B,\psi)$ there exists a $B(,\psi)$-generic simple supercuspidal representation $\pi'$ in the $G_\tx{ad}(F)$-orbit of $\pi$.

This completes the proof of the existence statement, and we turn to the uniqueness statement. Because of its importance, we will offer two separate proofs of quite different flavor. The first one uses elaborate analytic tools and is fairly short. The second one is longer but has the advantage of being elementary.

\subsection{An analytic proof of uniqueness} \label{sec:int_unia}
The uniqueness statement of Proposition \ref{pro:gen} can be readily deduced from the arguments of \cite[\S4]{DeRe10}.
We take as above an element of $\Pi_\phi$ given by $\textrm{c-Ind}\chi$. Fix as in \cite[\S2.6]{DeRe10} a bilinear form $\<\ ,\ \>$ on $\mf{g}(F)$ and an additive character $F \rw \C^\times$. Let $X \in \mf{g}(F)_{x,-\frac{1}{h}}$ correspond to the character $\chi$. We would like to quote \cite[Prop. 4.10]{DeRe10}. In this statement, the authors assume that the centralizer of $X$, which in our case is the image of an embedding of type (C), corresponds to a vertex in $\mc{B}(G,F)$. This is not the case for us, but one sees that all inputs in the proof -- Shelstad's result \cite{Sh89} and its interpretation by Kottwitz \cite[Prop. 4.2]{DeRe10}, the characterization of genericity via local character expansions due to Moeglin-Waldspurger \cite{MoWa87}, and Murnaghan-Kirillov theory as developed by Adler-DeBacker in \cite{AdDe04} -- are valid without this assumption. We see that if the representations
\[ \textrm{c-Ind}\chi \qquad \tx{and}\qquad \textrm{c-Ind}\chi' \]
are both generic with respect to the same Whittaker datum, then the elements $X,X'$ corresponding to $\chi,\chi'$ have $G(F)$-orbits which meet a given Kostant section. But if both representations belong to the $L$-packet $\Pi_\phi$, then $\chi$ and $\chi'$ are conjugate under $G_\tx{ad}(F)$, hence $X$ and $X'$ are conjugate under the same group. We conclude that then $X$ and $X'$ are already conjugate under $G(F)$ and the uniqueness statement follows.

\subsection{An elementary proof of uniqueness}
In this subsection only, we drop the requirement that $\tx{char}(k_F)$ is does not divide the Coxeter number of $G$.

Let
\[ \pi = \textrm{c-Ind}\chi \qquad \tx{and}\qquad \pi' = \textrm{c-Ind}\chi' \]
be two simple supercuspidal representations in the same orbit under $G_\tx{ad}(F)$, where $\chi,\chi' : Z(F)G(F)_{x,1/h} \rw \C^\times$ are affine generic characters. Assume that $\pi$ and $\pi'$ are generic with respect to the same fixed Whittaker datum $(B,\psi)$. Conjugating by $G_\tx{ad}(F)$ we may assume that $B$ contains a torus whose apartment contains $x$. Let $o$ be a special vertex contained in $\ol{C}$. Let $U=B_u$. Conjugating further by $G_\tx{ad}(F)$ we may assume that $\psi$ is non-trivial on $U_{o,0}=U(O_F)$ but trivial on $U_{o,1}$.

\begin{lem} \label{lem:d1} Let $(T,B)$ be a Borel pair in $G$, and $w \in \Omega(T,G)$ be an element which preserves the set $\Pi := \Delta \cup \{-\eta\}$, where $\eta$ is the highest root. Then
\begin{enumerate}
\item $w$ has a lift in $N(T_\tx{ad},G_\tx{ad})(F)$ of the same order as $w$.
\item Let $N : Q \rw Q$ be the norm map for the action of $w$ on the root lattice, and let $O=\<w\>(-\eta) \cap \Delta$. Then, for each $\alpha \in O$, $N(\alpha)$ belongs to the $\Z$-span of $\{N(\beta)| \beta \in \Delta \sm O\}$.
\end{enumerate}
\end{lem}
\begin{proof} The first statement is the content of \cite[Lemma 6.1]{GrLeReYu11}. To prove the second statement, one considers the matrix for the action on $Q$ of the possible elements $w$, given in the basis provided by the simple roots. This matrix can be obtained using the tables in \cite{Bou02}. The statement can then be read off from the matrix of the norm $N$.
\end{proof}

\begin{lem} \label{lem:d2}
\[ \tx{Hom}_{U(F)}(\pi,\psi) = \bigoplus_{t \in T(k_F)}\tx{Hom}_{U(O_F)}(\chi,\tx{Ad}(t)\psi) \]
\end{lem}

\begin{proof}
We apply \cite{Ku77} and examine the indexing set
\[ \Lambda:= U(F) \lmod G(F) / G(F)_{x,\frac{1}{h}}. \]
We claim that the map $\tx{N}(G,T)(F) \rw \Lambda$ is surjective. To that end, let $g \in G(F)$. We can write $g=utk$ with $u \in U(F)$, $t \in T(F)$ and $k \in G(O_F)$. Using the Bruhat decomposition for $G(k_F)$ we can write $k = vnv'y$ with $v,v' \in U(O_F)$, $n \in N(T,G)(O_F)$, and $y \in G(F)_{o,0+}$. Thus
\[ g=(uv)(tn)(v'y) \]
where $uv \in U(F)$, $tn \in N(T,G)(F)$ and $vy' \in G(F)_{x,0+}$. This proves the claim.

Next we claim that if $n \in N(F)$ projects to a non-trivial element in the affine Weyl group $N(F)/T(O_F)$, then
the restrictions of the characters $\chi$ and $\tx{Ad}(n)\psi$ to the group
\[ G(F)_{x,1/h} \cap n^{-1}U(F)n \]
cannot be equal. To that end, let $\alpha$ be a simple affine root with $\alpha(o)=0$. Then the restriction of $\tx{Ad}(n)\psi$ to $U(F)_{n^{-1}\alpha}$ is non-trivial while its restriction to $U(F)_{n^{-1}\alpha+1}$ is trivial. It follows that $\tx{Ad}(n)\psi$ and $\chi$ will not be equal if $n^{-1}\alpha$ is not an affine simple root. But if $n^{-1}\alpha$ is affine simple for every affine simple $\alpha$ with $\alpha(o)=0$, then $n^{-1}$ must preserve the alcove $C$. Since $G$ is simply-connected, this implies that the image of $n^{-1}$ in the affine Weyl group is trivial.

To complete the proof of the lemma it is enough to observe that the restriction of $G(F) \rw \Lambda$ to $T(O_F)$ factors through an injection $T(k_F) \rw \Lambda$.
\end{proof}

\begin{lem} \label{lem:d3} Let $\chi,\chi' : Z(F)G(F)_{x,1/h} \rw \C^\times$ be two affine generic characters. Assume that they are $G_\tx{ad}(F)_x$-conjugate and have the same restriction to $U(O_F)$. Then they are equal.
\end{lem}

\begin{proof} The restrictions of $\chi$ and $\chi'$ to $Z(F)$ are clearly the same, hence it is enough to focus on their restrictions to $G(F)_{x,1/h}$. These factor through the quotient
\begin{equation} \label{eq:quotx1} G(F)_{x,\frac{1}{h}}/G(F)_{x,\frac{2}{h}} = \bigoplus_{\alpha \in \Pi} U(F)_\alpha/U(F)_{\alpha+1}, \end{equation}
where $\alpha$ runs over the set $\Pi$ of affine simple roots. This quotient is a $k_F$-vector space. The image of $U(O_F)$ in it is
\[ \bigoplus_{\alpha \in \Delta} U(F)_\alpha/U(F)_{\alpha+1},\quad \tx{where}\quad \Delta := \{ \alpha \in \Pi| \alpha(o)=0 \}. \]
Thus, what we need to show is that the restrictions of $\chi$ and $\chi'$ to
\[ U(F)_{1-\eta}/U(F)_{2-\eta} \]
are the same, where $1-\eta$ is the unique element of $\Pi - \Delta$.

Let $h \in G_\tx{ad}(F)_x$ be an element element such that
\begin{equation} \chi' = \chi\circ\tx{Ad}(h^{-1}). \label{eq:chichi'} \end{equation}
Since $G(F)_{x,0+}$ acts transitively on the set of apartments containing $x$, and at the same time acts trivially on the above $k_F$-vector space, we may assume that $h$ preserves the apartment of $T$, i.e. $h \in N(T_\tx{ad},G_\tx{ad})_x$. This element acts on $\Pi$. Let $m$ be the size of the orbit of $1-\eta$. Then $h^m$ preserves both the alcove $C$ and the vertex $o$, thus $h^m \in T_\tx{ad}(O_F)$. This element acts on \eqref{eq:quotx1} through its image in $T_\tx{ad}(k_F)$ and preserves each individual summand. Hence, for each $\alpha \in \Pi$ there is a scalar $\upsilon_\alpha \in k_F^\times$ by which $h^m$ acts on $U(F)_\alpha/U(F)_{\alpha+1}$. The function $\alpha \mapsto \upsilon_\alpha$ is constant on the orbits of $h$, and the image of $h^m$ in $T_\tx{ad}(k_F)$ is equal to
\[ \prod_{\alpha \in \Delta} \check\omega_\alpha(\upsilon_\alpha) \]
where $\check\omega_\alpha$ is the coweight of the gradient of $\alpha$. We claim that this element is trivial.

Let $O \subset \Pi$ be the unique orbit of $\tx{Ad}(h)$ which is not contained in $\Delta$. Then \eqref{eq:chichi'} implies that for each $\alpha \in \Delta \sm O$, we have the following equality of characters of $U(F)_\alpha/U(F)_{\alpha+1}$:
\[ \chi=\chi\circ \tx{Ad}(h^m) = \chi\circ\upsilon_\alpha. \]
In other words, the endomorphism of $U(F)_\alpha/U(F)_{\alpha+1}$ given by multiplication by $\upsilon_\alpha-1$ has its image inside the kernel of $\chi$. Since $\chi$ is affine generic, this kernel is not the full group $U(F)_\alpha/U(F)_{\alpha+1}$, hence the endomorphism $\upsilon_\alpha-1$ is not bijective and thus must be zero. We have shown that $\upsilon_\alpha=1$ for all $\alpha \in \Delta \sm O$. According to Lemma \ref{lem:d1}, the element
\[ \prod_{\alpha \in \Delta} \check\omega_\alpha(\upsilon_\alpha) \]
belongs to the image of the norm for the action of $\tx{Ad}(h)$ on $T_\tx{ad}(F)$, and thus for $\alpha \in O \cap \Delta$, the element $\upsilon_\alpha$ is a product of powers of the elements $\upsilon_\beta$ for $\beta \in \Delta \sm O$. This shows that indeed $\upsilon_\alpha=1$ for all $\alpha \in \Pi$ and the proof is complete.

\end{proof}

The uniqueness statement in Proposition \ref{pro:gen} now follows easily from these Lemmas -- if both $\pi$ and $\pi'$ are generic, then Lemma \ref{lem:d2} implies that there exist $t,t' \in T(k_F)$ such that
\[ \tx{Hom}_{U(O_F)}(\chi,\tx{Ad}(t)\psi) \neq 0 \quad \tx{and}\quad \tx{Hom}_{U(O_F)}(\chi',\tx{Ad}(t')\psi) \neq 0.  \]
Replacing $\chi$ by $\tx{Ad}(t^{-1})\chi$ and $\chi'$ by $\tx{Ad}(t'^{-1})\chi'$ does not change $\pi$ and $\pi'$. But now Lemma \ref{lem:d3} implies that $\chi=\chi'$.

\section{Unramified extensions} \label{sec:unram}

The conjecture of Gross and Reeder that simple supercuspidal representations should correspond to simple wild parameters \cite[9.5]{GrRe10} has as an input the assumption of a certain natural compatibility of this correspondence with unramified extensions of the base field $F$. In this section we are going to show that the correspondence we have just constructed satisfies such a compatibility.

Let $\phi : W_F \rw \hat G$ be a simple wild parameter. Given a finite unramified extension $\tilde F$ of $F$, we can restrict the parameter $\phi$ to $W_{\tilde F}$. Call this restriction $\tilde\phi$. It is still a simple wild parameter.
Consider the norm map
\[ \tx{Cent}(\tilde\phi,\hat G) \rw \tx{Cent}(\phi,\hat G),\qquad g \mapsto \prod_{w \in W_F/W_{\tilde F}} \tx{Ad}(\phi(w))g. \]
We obtain the following diagram
\begin{diagram}
X^*(\tx{Cent}(\phi,\hat G))&&\rTo&&X^*(\tx{Cent}(\tilde\phi,\hat G))\\
\dTo&&&&\dTo\\
\Pi_\phi&&&&\Pi_{\tilde\phi}
\end{diagram}
where the horizontal map is the dual of the norm map and the vertical maps are the bijections constructed in Section \ref{sec:int}, each depending on the choice of a Whittaker datum. There is an obvious way to make both choices coherent -- if we have chosen a Whittaker datum $(B,\psi)$ for $G$, then we may take as a Whittaker datum for $G\times\tilde F$ the pair $(B \times \tilde F,\psi\circ N)$ where $N$ is the norm map for the action of $\Gamma(\tilde F/F)$ on the quotient
\[ U(\tilde F)/[U,U](\tilde F) \cong \bigoplus_{\alpha \in \Delta} U_\alpha(\tilde F). \]

Given $\rho \in X^*(\tx{Cent}(\phi,\hat G))$, let $\tilde\rho \in X^*(\tx{Cent}(\tilde\phi,\hat G))$ be its image under the dual norm map. We can consider the simple supercuspidal representation $\pi$ of $G(F)$ corresponding to the pair $(\phi,\rho)$, as well as the simple supercuspidal representation $\tilde\pi$ of $G(\tilde F)$ corresponding to the pair $(\tilde\phi,\tilde\rho)$.

\begin{pro} If we write $\pi = \textrm{c-Ind}\chi$, where $\chi : Z(F)G(F)_{x,1/h} \rw \C^\times$ is an affine generic character, then $\tilde\pi = \textrm{c-Ind}\tilde\chi$, where $\tilde\chi : Z(\tilde F)G(\tilde F)_{x,1/h} \rw \C^\times$ is obtained by composing $\chi$ with the norm map for the action of $\Gamma(\tilde F/F)$ on the abelian group $Z(\tilde F)G(\tilde F)_{x,1/h}/G(\tilde F)_{x,2/h}$.
\end{pro}

Recall that the construction of the $L$-packet associated to $\phi$ in Section \ref{sec:pack} associates to $\phi$ a $G_\tx{ad}(F)$-conjugacy class of affine generic characters, then constructs a simply transitive action of $X^*(\tx{Cent}(\phi,\hat G))$ on the set of representations induced from these characters, and finally provides a base point of the set of these representations given by the choice of a Whittaker datum. The proof of the proposition will follow this structure -- we will first show that if $\chi$ is an affine generic character in the $G_\tx{ad}(F)$-orbit associated to $\phi$, then $\chi\circ N$ is an affine generic character in the $G_\tx{ad}(\tilde F)$-orbit associated to $\tilde\phi$. Next we will show that if $\rho\cdot\textrm{c-Ind}\chi=\textrm{c-Ind}\chi'$, then $\tilde\rho \cdot \textrm{c-Ind}[\chi\circ N] = \textrm{c-Ind}[\chi'\circ N]$, where $\cdot$ denotes the simple transitive actions we alluded to. Finally, we will show that $\textrm{c-Ind}\chi$ and $\textrm{c-Ind}(\chi\circ N)$ are generic with respect to two choices of Whittaker data which are coherent in the sense described above.

\begin{proof}
Let $j : S \rw G$ be an embedding of type (C) and $\chi_S : S(F) \rw \C^\times$ be a character, such that the $G_\tx{ad}(F)$-conjugacy class of the pair $(\chi_S,j)$ is associated to $\phi$ by the construction of Section \ref{sec:pack}.
Recall that $\chi_S$ was obtained from a factoring of $\phi$ as
\[ W_F \stackrel{\phi_S}{\lrw} {^LS} \stackrel{^Lj}{\lrw} \hat G, \]
where $^Lj$ was constructed from a choice of tamely-ramified $\chi$-data $\{\chi_\lambda|\lambda \in R(S,G)\}$.

Each $\chi_\lambda$ is a character on the multiplicative group of $\ol{F}^{[\Gamma_F]_\lambda}$, where $[\Gamma_F]_\lambda$ is the stabilizer of $\lambda$ for the action of $\Gamma_F$ on $R(S,G)$. Let $N_\lambda$ denote the norm map for the extension
\[ \ol{F}^{[\Gamma_{\tilde F}]_\lambda}/\ol{F}^{[\Gamma_F]_\lambda}.\]
Then one can check that the set $\{\chi_\lambda\circ N_\lambda|\lambda \in R(S,G)\}$ satisfies the axioms of $\chi$-data for the action of $\Gamma_{\tilde F}$ on $R(S,G)$. This $\chi$-data is again tamely-ramified and can be used to produce an embedding ${^L\tilde j}:\hat S \rtimes W_{\tilde F} \rw \hat G$. One checks that the following diagram commutes
\begin{diagram}
W_F&\rTo^{\phi_S}&\hat S \rtimes W_F&\rTo^{^Lj}&\hat G\\
\uInto&&\uInto&\ruTo>{^L\tilde j}\\
W_{\tilde F}&\rTo^{\tilde\phi_S}&\hat S\rtimes W_{\tilde F}
\end{diagram}
In particular $\tilde\phi = {^L\tilde j}\circ\tilde\phi_S$. By the Langlands correspondence for tori, the character of $S(\tilde F)$ associated with $\tilde \phi_S$ is $\chi_S \circ N$, where $N : S(\tilde F) \rw S(F)$ is the norm map for $\tilde F/F$. We conclude that the $G_\tx{ad}(\tilde F)$-conjugacy class of the pair $(\chi_S \circ N,j \times \tilde F)$ is associated to $\tilde\phi$.

Let $x$ the point in $\mc{B}(G,F)$ associated to $j$. It is also the point associated to $j \times \tilde F$ by Fact \ref{fct:typecbc}. It is easy to see that we have the commutative diagram
\begin{diagram}
Z(F)G(F)_{x,\frac{1}{h}}/G(F)_{x,\frac{2}{h}}&\rTo&S(F)/S(F)_\frac{2}{h}\\
\dInto&&\dInto\\
Z(\tilde F)G(\tilde F)_{x,\frac{1}{h}}/G(\tilde F)_{x,\frac{2}{h}}&\rTo&S(\tilde F)/S(\tilde F)_\frac{2}{h}
\end{diagram}
where the horizontal maps are the homomorphisms \eqref{eq:homgen}, the top one associated to $j$ and the bottom one associated to $j \times \tilde F$. Moreover, the bottom one is $\Gamma(\tilde F/F)$-equivariant.
The affine generic character $\chi$ associated to the pair $(\chi_S,j)$ is the composition of $\chi_S$ and the top map, while the affine generic character $\tilde\chi$ associated to the pair $(\chi_S\circ N,j \times \tilde F)$ is the composition of $\chi_S\circ N$ and the bottom map. We conclude that if the $G_\tx{ad}(F)$-orbit of the affine generic character $\chi$ is associated to $\phi$, then the $G_\tx{ad}(\tilde F)$-orbit of the affine generic character $\chi\circ N$ is associated to $\tilde\phi$.

To show that the simply transitive actions are compatible, consider the diagram
\begin{diagram}
X^*(\tx{Cent}(\phi,\hat G))&&\rTo&&X^*(\tx{Cent}(\tilde\phi,\hat G))\\
\dEquals&&&&\dEquals\\
X_*(S)_{\Gamma_{F}}&&\rTo&&X_*(S)_{\Gamma_{\tilde F}}\\
\dEquals&&&&\dEquals\\
H^1(F,S)&&\rTo&&H^1(\tilde F,S)\\
\dEquals&&&&\dEquals\\
\frac{G_\tx{ad}(F)}{G(F)S_\tx{ad}(F)}&&\rTo&&\frac{G_\tx{ad}(\tilde F)}{G(\tilde F)S_\tx{ad}(\tilde F)}
\end{diagram}
where the top vertical maps are induced by $^Lj$ and $^L\tilde j$, the middle vertical maps are the Tate-Nakayama isomorphisms, and the bottom vertical maps arise from the connecting homomorphisms. The top horizontal map is the dual norm map, while the bottom horizontal map is induced by the inclusion $G_\tx{ad}(F) \rw G_\tx{ad}(\tilde F)$. From this it follows that
if $\rho\cdot\textrm{c-Ind}\chi=\textrm{c-Ind}\chi'$, then $\tilde\rho \cdot \textrm{c-Ind}[\chi\circ N] = \textrm{c-Ind}[\chi'\circ N]$, which was to be shown.

To complete the proof of the Proposition, choose a Borel subgroup $B$ containing a maximal torus whose apartment contains $x$, and let $\psi$ be a character of $B_u(F)$ whose restriction to $B_u(O_F)$ coincides with $\chi$. Then $\pi=\textrm{c-Ind}\chi$ is $(B,\psi)$ generic, and at the same time $\tilde\pi=\textrm{c-Ind}\tilde\chi$ is $(B\times \tilde F,\psi\circ N)$-generic. Thus $\pi$ and $\tilde\pi$ correspond to the same element of $X^*(\tx{Cent}(\phi,\hat G))$.

\end{proof}

\section{Stability} \label{sec:stab}

Let $x \in \mc{B}(G,F)$ be the barycenter of an alcove,
\[ \chi : Z(F)G(F)_{x,\frac{1}{h}}/G(F)_{x,\frac{2}{h}} \rw \C^\times \]
be an affine generic character, and $\pi$ the corresponding simple supercuspidal representation. We are going to denote its Harish-Chandra character function by $\Theta_\pi$. Let $\Pi$ be the $L$-packet containing $\pi$ and put
\[ S\Theta_\Pi = \sum_{\pi \in \Pi} \Theta_\pi. \]
One of the conjectural properties of $L$-packets is that the function $S\Theta_\Pi$ should be stable. That is, if $\gamma,\gamma' \in G(F)$ are regular semi-simple elements for which there exists $g \in G(\ol{F})$ that conjugates the one to the other, than
\[ S\Theta_\Pi(\gamma) = S\Theta_\Pi(\gamma'). \]
This function can be studied independently of our construction of the $L$-packet $\Pi$, if one takes $\Pi$ to simply be the orbit of $\pi$ under $G_\tx{ad}(F)$. We have learned from Stephen DeBacker that one can use the results of Adler-Spice \cite{AdSp09} to prove the stability of $S\Theta_\Pi$ under the assumption that the residue field $k_F$ has sufficiently large characteristic. In this section, we are going to provide a different indication of stability. We will show that without additional assumptions on $k_F$ (beyond the standing assumption that $\tx{char}(k_F)$ does not divide the Coxeter number of $G$), one can show that on a certain open set of regular semi-simple elements, the function $S\Theta_\Pi$ is atomically stable -- i.e. it is stable, and no sum of characters of a proper subset of $\Pi$ is stable. This result is an indication that the stability of $\Theta_\Pi$ (at all regular semi-simple elements) holds without further restrictions on $k_F$, but we have not been able to find a proof for this statement yet.

\begin{pro} Let $j: S' \rw G$ be any embedding of type (C), and let $\gamma \in j(S')(F)$ be a regular element. If $\gamma' \in G(F)$ is stably conjugate to $\gamma$, then $S\Theta_\Pi(\gamma)=S\Theta_\Pi(\gamma')$. Moreover, if $\Pi'$ is a proper subset of $\Pi$, then $S\Theta_{\Pi'}(\gamma) \neq S\Theta_{\Pi'}(\gamma')$.
\end{pro}
\begin{proof}
Let $\dot\chi : G(F) \rw \C$ denote the function which equals $\chi$ on $Z(F)G(F)_{x,1/h}$ and equals zero outside of this subgroup. This function is an element of the induced representation $\pi$, and since the evaluation map $f \mapsto f(1)$ is a smooth functional on $\pi$, the function $\dot\chi$ is also a matrix coefficient for $\pi$. According to \cite[Lemma 22.1]{HC99}, we have
\[ \Theta_\pi(\gamma) = \deg(\pi;dx) \int_{G(F)/Z(F)} \dot\chi(x\gamma x^{-1})dx. \]
By construction, the finite abelian group $G_\tx{ad}(F)/G(F)$ acts transitively on $\Pi$. It follows that
\[ S\Theta_\pi(\gamma) = C\int_{G_\tx{ad}(F)} \dot\chi(x\gamma x^{-1})dx, \]
where $C$ is the positive constant given by the ratio of $\deg(\pi;dx)$ and the size of the stabilizer of $\pi$ in $G_\tx{ad}(F)/G(F)$. The claim now follows from Corollary \ref{cor:adstab}.
\end{proof}

Note that the embedding $j$ in the statement of the proposition need not be the one used in the construction of the packet $\Pi$.

{\small
Tasho Kaletha\\
tkaletha@math.ias.edu\\
School of Mathematics, Institute for Advanced Study, \\
Einstein Drive, Princeton, NJ 08540
}
\end{document}